\documentclass[12pt]{article}
\usepackage{amsmath,amsfonts,amsthm,amscd,amssymb,graphicx}

\usepackage{amssymb}

\def\rit{{\Bbb R}}
\def\cit{{\Bbb C}}

\def\tit{{\Bbb T}}

\def\eps{\varepsilon}

\def\Remark{{\noindent \it Remark: }}
\def\Remarks{{\noindent \it Remarks: }}

\newtheorem{theorem}{Theorem}[section]
\newtheorem{lemma}[theorem]{Lemma}
\newtheorem{e-proposition}[theorem]{Proposition}

\newtheorem{e-definition}[theorem]{Definition\rm}

\newtheorem{theoreme}{Th\'eor\`eme}[section]

\newtheorem{proposition}[theoreme]{Proposition}

\def\og{\leavevmode\raise.3ex\hbox{$\scriptscriptstyle\langle\!\langle$~}}
\def\fg{\leavevmode\raise.3ex\hbox{~$\!\scriptscriptstyle\,\rangle\!\rangle$}}

\def\beq{\begin{equation}}
\def\eeq{\end{equation}}

\begin{document}

\centerline{\Large \bf Asymptotic behavior of solutions }

\bigskip

\centerline{\Large \bf of the linearized Euler  equations near a shear layer}

\bigskip

\centerline{D. Bian\footnote{School of Mathematics and Statistics, Beijing Institute of Technology, $100081$ Beijing, China. Email: 
biandongfen@bit.edu.cn}, 
E. Grenier\footnote{Unit\'e de Math\'ematiques Pures et Appliqu\'ees, CNRS  UMR $5669$, Ecole Normale Sup\'erieure de Lyon,
$46$ all\'ee d'Italie, $69007$ Lyon, France. Email: Emmanuel.Grenier@ens-lyon.fr}}



\subsubsection*{Abstract}


In this article, thanks to a  new and detailed study of the Green's function of Rayleigh equation 
near the extrema of the velocity of a shear layer, 
we obtain optimal bounds on the asymptotic behaviour of solutions to the linearized incompressible Euler equations 
both in the whole plane, the half plane and the periodic case,
and improve the description of the so called "vorticity depletion property" 
discovered by F. Bouchet and H. Morita \cite{Bouchet-Morita-2010} by putting into light
a localization property of the solutions of Rayleigh equation near an extremal velocity.


\section{Introduction}


In this paper we focus on the asymptotic behaviour of  the solutions of the linearized incompressible Euler equations 
near a shear layer $U(y)$. More precisely we study the large time behaviour of solutions $v(t,x)$ to
\beq \label{Eulerlin1} 
\partial_t v + (U \cdot \nabla) v + (v \cdot \nabla) U  + \nabla p = 0,
\eeq
\beq \label{Eulerlin2}
\nabla \cdot v = 0 .
\eeq
Though similar results are valid in the whole space and in the periodic case, and can be obtained using the same ideas, 
we focus on the half plane case, namely $y \ge 0$,
together with the classical boundary condition
\beq \label{Eulerlin3} 
v_2 = 0 \qquad \hbox{for} \qquad y = 0.
\eeq
In all this paper,  $U(y) = (U_s(y),0)$ is a given shear layer, and $U_s(y)$ is a $C^\infty(\rit^+)$ function, 
satisfying $U_s(0) = 0$, $U_s'(0) \ne 0$, 
 such that $U_s(y)$ converges exponentially fast to some given constant $U_+$ as
 $y \to + \infty$, and such that $U_s''(y)$ converges exponentially fast to $0$  as $y$ goes to $+\infty$.

We will moreover assume that there exists at most a finite number of $y$ such that $U_s'(y) = 0$,
that, for all these $y$,  $U_s''(y) \ne 0$, and that the corresponding values $U_s(y)$ are all different.
These values of $y$, if they exist, will be called "extremal layers" and denoted $y_{extr}^l$ for $l = 1,\dots,L$.
The corresponding velocities $c_{extr}^l = U_s(y_{extr}^l)$ will be called "extremal velocities".

These assumptions include, for instance,  $U_s(y) = U_+ (1 - e^{-y})$  and $U_s(y) = y e^{-y}$.

\medskip

In the recent years, there have been many new mathematical results and breakthroughs on the linear and nonlinear stability of 
particular shear flows 
(like Couette flow, Poiseuille flow, Kolmogorov flow, or shear flows near Couette) for the incompressible Euler equations 
in the periodic domains $\mathbb{T}\times [0,1]$, $\mathbb{T}\times \mathbb{R}$ or $\mathbb{T}^2$, in the wake of the seminal
work of J. Bedrossian and N. Masmoudi \cite{Bedrossian-Masmoudi-2015} on Couette flow,  see for example 
\cite{Bianchini-Zelati-Dolce,Deng-Zilinger-2020,Ionescu-Jia-AM,Ionescu-Jia-CMP-2020,Ionescu-Iyer-Jia,Wei-Zhang-Zhao-CPAM-2018,Wei-Zhang-Zhao-APDE-2019,Wei-Zhang-Zhao-AdvM-2020}. 
 
\medskip

The study of the asymptotic behavior of solutions to (\ref{Eulerlin1},\ref{Eulerlin2},\ref{Eulerlin3}) for general shear layer profiles
$U_s(y)$ has been pioneered in physics by F. Bouchet and H. Morita 
in \cite{Bouchet-Morita-2010}. They conjectured that, in the absence of unstable eigenvalues, the vertical velocity should decay like
$\langle t \rangle^{-2}$, the horizontal velocity like $\langle t \rangle^{-1}$, and that the vorticity should remain of order $O(1)$.
They also noticed that, asymptotically, the horizontal Fourier transform $\omega_\alpha(t,y)$ of the vorticity $\omega(t,x,y)$ should satisfy
\beq \label{behavioromega}
\omega_\alpha(t,y) = \omega^\infty_\alpha(y) e^{- i \alpha U_s(y) t} + O(\langle t^{-\gamma} \rangle)
\eeq
for some limiting vorticity profile $\omega^\infty_\alpha(y)$ and for some unspecified positive exponent $\gamma$. They moreover
conjecture that 
\beq \label{behavioromega2}
\omega^\infty_\alpha(y_{extr}^l) = 0 
\eeq
for any $1 \le l \le L$, a phenomena they called "depletion property". In other words, in large times, the vorticity should be small near each $y_{extr}^l$.

The question of the value of $\gamma$, together with the description of $\omega_\alpha^\infty$ near $y_{extr}^l$ 
are left open in \cite{Bouchet-Morita-2010}. The aim of this paper is to address in details these two questions and to obtain optimal bounds
on the decay of the horizontal and vertical velocities.

Previously, in \cite{Wei-Zhang-Zhao-APDE-2019},   for symmetric flows, D. Wei, Z. Zhang, W. Zhao
 have proven the $\langle t \rangle^{-2}$ decay of the vertical velocity  
 and the $\langle t \rangle^{-1}$ decay of the horizontal one in $L^2$ norm, together
with (\ref{behavioromega2}), but left open the  studies of $\gamma$ and of 
 $\omega^\infty_\alpha$ near the extremal layers. The authors have also investigated special flows, like Kolmogorov flows
 \cite{Wei-Zhang-Zhao-AdvM-2020} or the  "$\tanh$" flow \cite{Zhang3}.
 
Later,  in \cite{Ionescu-Iyer-Jia}, A. Ionescu, S. Iyer and H. Jia  extended the proof of the decay of the velocities
to general flows, and made progresses on the value of $\gamma$
and on the behavior of $\omega_\alpha^\infty$ near the extremal layers.
 They in particular proved that the vorticity $\omega(t,x,y)$ may be split 
 into a "local part" $\omega_{loc}$ and a "nonlocal part" $\omega_{nloc}$ with, near each $y_{extr}^l$,
\beq \label{behavioromega3}
| \omega_{loc}(t,x,y)) | \lesssim | y - y_{extr}^l|^{7/4} 
\eeq
and
\beq  \label{behaviorom}
| \omega_{nloc}(t,x,y) | \lesssim \langle t \rangle^{-7/8},
\eeq
and conjectured that the same inequalities should be true with $2^-$ and $1^-$ instead of $7/4$ and $7/8$.

\medskip

 As we will see below, the $L^2$ decay of the velocities hides a slightly slower decay 
 occuring close to the extremal layer. This slower decay can be seen if we look for pointwise, namely  $L^\infty$ estimates,
 but is hidden if we use energy norms.
  In this paper we address $L^\infty$ norms, which requires to correct Bouchet and Morita's conjectures by a logarithmic factor,
not described in  \cite{Bouchet-Morita-2010}.
 
\medskip

The aim of this paper is to  give an elementary and short proof of F. Bouchet and H. Morita's computations, together with optimal 
bounds in (\ref{behavioromega}), (\ref{behavioromega3}) and (\ref{behaviorom}),
and together with  optimal pointwise bounds on the decay of the velocities.

\medskip

The proof  follows a direct approach. The main new ingredient is an
explicit construction of the Green's function of Rayleigh equation, thanks to a detailed  
 description of the solutions of this equation, even near points where $U_s'$ vanishes.
 We in particular set into light a "localization property" of the solutions of Rayleigh equation near extremal velocities.
This allows us to  simplify 
\cite{Ionescu-Iyer-Jia,Wei-Zhang-Zhao-APDE-2019},
to  improve the description of the "vorticity depletion", to prove the optimal bounds conjectured in \cite{Ionescu-Iyer-Jia},
and also to add a boundary, possible unstable eigenvalues as well as simple embedded eigenvalues.

\medskip

Up to the best of our knowledge, the central part of the current work, 
namely the detailed description of the solutions of Rayleigh equation near the extremum of
the velocity is completely new.

\subsubsection*{Notation}

We say that $f \lesssim g$ if there exists a constant $C$ such that $| f | \le C |g|$.
We define 
$$
\langle t \rangle = 1 + | t |.
$$
We denote by $Range(U_s)$ the closure of the set $\{U_s(y), y \in \rit^+\}$.
We will say that two functions $f$ and $g$ are of the same order if both $f / g$ and $g / f$ are bounded.


\section{Setup  and main results}


Our method also  applies to the $x$ - periodic cases $\tit^2$ or  $\tit \times \rit$, and to
 $\rit \times \rit$ and $\rit \times \tit$. But for the sake of simplicity, we only state the results in the half plane case $\rit \times \rit^+$.
We first restate the assumptions on the shear layer profile $U_s(y)$ as follows.

\medskip

\noindent Assumption (A1): 
{\it $U_s(y)$ is a $C^\infty$ function such that $U_s(0) = 0$, $U_s'(0) \ne 0$,  such that $U_s(y)$ 
converges exponentially fast to some constant $U_+$ at infinity, and such that $U_s''(y)$ converges exponentially fast to $0$
as $y \to + \infty$.}

\medskip

\noindent Assumption (A2): 
{\it Either   $U_s'(y)$ never vanishes on $\rit^+$, or the equation $U_s'(y) = 0$
has a finite number of solutions, denoted by $(y_{extr}^l)_{1 \le l \le L}$,  and called extremal layers.
In this case, the corresponding extremal velocities
$c_{extr}^l  = U_s(y_{extr}^l)$
are all different,  and, for any $1 \le l \le L$,  $U_s''(y_{extr}^l) \ne 0$.}

\medskip

We now recall the classical setting \cite{Reid} and introduce Rayleigh equation.
Let $L_{E}$ be the linearized Euler  operator near the shear layer profile $U_s$, namely
\beq \label{linearNS}
L _E \, v  = (U \cdot \nabla) v + (v\cdot \nabla) U + \nabla p,
\eeq
with $\nabla \cdot v = 0$ and $v_2 = 0$ on the boundary. A classical way to compute the solution $v(t)$ 
of linearized Euler equations is to use contour integrals. More precisely, as $L_E$ is closed and sectorial \cite{Haase}, the solution $v(t)$
of 
$$
\partial_t v + L_E \, v = 0,
$$
 with initial data $v_0$, is given through  the Dunford's contour integral \cite{Dunford,Haase},
\beq \label{contour1}
v(t) = {1 \over 2 i \pi} \int_\Gamma e^{\lambda t} (L_{E} + \lambda)^{-1} v_0 \; d\lambda ,
\eeq
in which $\Gamma$ is a contour "on the right" of the spectrum.
 We are therefore led to study the resolvent of $L_{E}$, namely to
study the equation
\beq \label{resolvant}
(L_{E} + \lambda) v = f,
\eeq
where $f$ is a given forcing term and $\lambda$ a complex number. 
To take advantage of the divergence free condition, we introduce the stream function $\psi$ of $v$ and take the Fourier
transform in $x$, which leads to look for solutions of (\ref{resolvant}) of the form
$$
v = \nabla^\perp \Bigl( e^{i \alpha (x - ct) } \psi(y) \Bigr),
$$
where $\nabla^\perp = (\partial_y,- \partial_x)$ and where $c$ is defined by
$$
\lambda = - i \alpha c,
$$
according to the traditional notations. 
We also take the Fourier  transform of the forcing term $f$
$$
f = \Bigl( f_1(y),f_2(y) \Bigr) e^{i \alpha x } .
$$
Taking the curl of (\ref{resolvant}), we then get the classical Rayleigh equation
\beq \label{Ray}
Ray_{\alpha,c}(\psi) =  (U_s - c)  (\partial_y^2 - \alpha^2) \psi - U_s''  \psi  
 =  i {\nabla \times f \over \alpha},
\eeq
where
$$
\nabla \times (f_1,f_2) = i \alpha f_2 - \partial_y f_1,
$$
together with the boundary condition 
\beq \label{condOrr}
\psi(0) = 0.
\eeq
Note that the vorticity is 
$$
\omega = - (\partial_y^2 - \alpha^2) \psi.
$$
When $f = v_0$, the right hand side of (\ref{Ray}) is simply $i \alpha^{-1} \omega_\alpha^0$,
where $\omega_\alpha^0$ is the Fourier transform in $x$ of the initial vorticity
$\omega_0 = \nabla \times v_0$.

Let us first consider Rayleigh equation without forcing term and without boundary condition, namely 
\beq \label{Rayleighwithout1}
(U_s - c) (\partial_y^2 - \alpha^2) \psi - U_s'' \psi = 0 .
\eeq
At infinity, $U_s''(y)$ decays exponentially fast to $0$. As a consequence, 
for a given complex number $c$ outside of $Range(U_s)$,   there exist two solutions $\psi_{+,\alpha,c}(y)$ 
and $\psi_{-,\alpha,c}(y)$ which asymptotically behave like $e^{|\alpha| y}$ and $e^{- |\alpha| y}$ when $y$ goes to $+\infty$.
These solutions are however not well defined if $c \in Range(U_s)$.
As a consequence, the continuous spectrum $\sigma_C$ of linearized Euler equations is the $Range(U_s)$,
and the point spectrum $\sigma_P$ satisfies
$$
\sigma_P(U_s)  \backslash  Range(U_s) =  \Bigl\{ c \in \cit \backslash Range(U_s)  \; | \;    \psi_{-,\alpha,c}(0) = 0 \Bigr\}.
$$
The first step of the analysis is to investigate whether there exists so called embedded eigenvalues,
namely eigenvalues in $\sigma_C$. For this we need to extend $\psi_{-,\alpha,c}(0)$ to real values of $c$.

\begin{proposition} \label{extension} (Extension of the dispersion relation) \\
Under the assumptions (A1) and (A2), 
 $\psi_{-,\alpha,c}(0)$ may be extended by continuity to $\rit^+ - \{ (c_{extr}^l)_{1 \le l \le L} \}$.
\end{proposition}

As will be clear in the proof, $\psi_{-,\alpha,c}(0)$ may be singular at the extremal layers, and even become infinite.

The embedded eigenvalues are the zeros of $\psi_{-,\alpha,c}(0)$ which are in $Range(U_s)$.
We now add two assumptions on the embedded eigenvalues.

\medskip

\noindent Assumption (A3): {\it $\psi_{-,\alpha,c}(0)$ is bounded away from $0$ in a neighbourhood of $\{ (c_{extr}^l)_{1 \le l \le L} \}$}.

\medskip

\noindent Assumption (A4): {\it There exists at most a finite number of embedded eigenvalues 
$(c_{embedded}^k)_{1 \le k \le K}$, namely a finite number of solution to $\psi_{-,\alpha,c}(0) = 0$ with
$c \in Range(U_s) -   \{ (c_{extr}^l)_{1 \le l \le L} \}$, and they are all  simple.}

\medskip

Assumption (A3)  is a way to say that the extremal values are not embedded
eigenvalues, which would uselessly complicate the analysis for a very particular case 
(there is no evidence that such a case could even exist).
With these assumptions we can now extend $\psi_{\pm,\alpha,c}$ to  real values of $c$.

\begin{proposition} \label{localization} (localization principle) \\ 
Let us assume (A1), (A2), (A3) and (A4).
Let $\alpha$ be fixed. Then the two solutions $\psi_{\pm,\alpha,c}(y)$, with Wronskian $1$, can be extended by continuity to $\rit^+ - \{ (c_{extr}^l)_{1 \le l \le L} \}$
in such a way that $\psi_{-,\alpha,c}(y) \to 0$ as $y \to + \infty$.
Moreover, if $c$ is away from the extremal velocities $(c_{extr}^l)_{1 \le l \le L}$ and away from $U_+$,
\beq \label{bbbb1}
| \psi_{\pm,\alpha,c} (y) | \lesssim e^{\pm | \alpha | y} .
\eeq
If, on the contrary, $c$ is close to a extremal value $c_{extr}^l$ for some $1 \le l \le L$,
then for any $\eps > 0$ arbitrarily small,  in the vicinity of $c_{extr}^l$,  for $y > y_{extr}^l + \eps$,
\beq \label{bbbb2}
| \psi_{-,\alpha,c} (y) | \lesssim  e^{- | \alpha | y}, \qquad | \psi_{+,\alpha,c}(y) | \lesssim e^{+ | \alpha | y},
\eeq
and for $y < y_{extr}^l - \eps$,
\beq \label{bbbb3}
| \psi_{-,\alpha,c}(y) | \lesssim  | c - c_{extr}^l |^{-3/2}, \qquad
| \psi_{+,\alpha,c}(y) | \lesssim | c - c_{extr}^l |^{+3/2}.
\eeq
\end{proposition}

This Proposition may be interpreted as a "localization principle". As $c \to c_{extr}^l$, 
$(c - c_{extr}^l)^{3/2} \psi_{-,\alpha,c}(y)$ goes to $0$ for $y > y_{extr}^l$ and is of order $1$ for $y < y_{extr}^l$, 
namely get "localised" on $y < y_{extr}$ as $c \to c_{extr}^l$. On the contrary,
$\psi_{+,\alpha,y}(y)$ goes to $0$ for $y < y_{extr}$ and is of order $1$ for $y > y_{extr}^l$.

The localization property is the key ingredient of the vorticity depletion phenomenum.
As $c$ goes to $c_{extr}^l$, the "information" coming from the left of the extremal point  gets blocked at
$y_{extr}^l$ and does not affect the flow for $y > y_{extr}^l$: "what is on the left remains on the left", and conversely
for the right.

We now turn to the asymptotic behavior of solutions to linearized Euler equations.

\begin{theorem} \label{growthRayleigh} (Asymptotic behaviour of the solutions) \\
Let us assume (A1), (A2), (A3),and (A4).
Let us assume that the horizontal Fourier transform $\omega_\alpha^0(x)$ of the initial vorticity
is a $C^2$ function which decays exponentially fast at infinity.
Let 
\beq \label{Theta}
\Theta(x) = 1 + 1_{|x| \le 1} | \log | x| | .
\eeq
Then we have, for any fixed $\alpha \ne 0$,
\beq \label{result1}
\psi_\alpha(t,y) = \psi^{modes}_\alpha(t,y) + \psi_\alpha^{decay}(t,y),
\eeq
where $\psi^{modes}_\alpha$ is a linear combination of eigenvenvectors of Rayleigh equation, with corresponding
eigenvalues in $\sigma_P \cap \{ \Im c \ge 0 \}$, and where 
\beq \label{correctebound1}
| \psi_\alpha^{decay}(t,y) | \lesssim \langle \alpha t \rangle^{-1}  \Pi_{l=1}^L \Theta( y - y_{extr}^l)
\eeq
and
\beq \label{correctedbound2}
 | \partial_y \psi_\alpha^{decay}(t,x,y) | 
\lesssim \langle \alpha t \rangle^{-1}  \Pi_{l=1}^L \Theta( y - y_{extr}^l)  .
\eeq
 Moreover, there exists $\omega_\alpha^\infty(t,y)$ such that
 \beq \label{result4}
 \omega_\alpha(t,y) = \omega_\alpha^{modes}(t,y) + \omega_\alpha^\infty(y) e^{- i \alpha U_s(y) t} + \omega_\alpha^{remain}(t,y),
 \eeq
 where $\omega_\alpha^{modes}(t,y)$ is the vorticity corresponding to $\psi_\alpha^{modes}(t,y)$, where $\omega_\alpha^\infty(y)$
 is such that   for any $1 \le l \le L$,
 \beq \label{deep1}
 \omega_\alpha^\infty(y_{extr}^l) = 0,
 \eeq
 and such that,   near  each $y_{extr}^l$,
 \beq \label{deep2}
 | \omega_\alpha^\infty(y) | \lesssim | y - y_{extr}^l |^2 \, |\log | y - y_{extr}^l | \, | ,
 \eeq
and where
\beq \label{omegaremain}
| \omega_\alpha^{remain}(t,y) | \lesssim   \langle \log(\alpha t) \rangle^{-1}  \Pi_{l=1}^L \Theta( y - y_{extr}^l).
\eeq
\end{theorem}

The factor $\langle \alpha t \rangle^{-1}$ in (\ref{correctebound1}) comes from the boundary layer. In the periodic case or in the whole space case,
the decay is $\langle \alpha t \rangle^{-2}$ as in  \cite{Bouchet-Morita-2010}.

 Note  that structures of wavenumbers $\alpha$ evolve under time scales of order $\alpha^{-1}$,
  which leads to the $\alpha t$ homogeneity  of the decay.
  Although we do not detail this points in the proof, the estimates are uniform in $\alpha$.
All the functions we consider decay exponentially fast at infinity. We thus only focus on their behaviour for bounded $y$ 
and do not detail their behaviour at infinity. 

Higher order degenerate points, namely points with $U_s'(y) = U_s''(y) = 0$ and $U_s'''(y) \ne 0$ and so on, could be handled using the same technics.

\medskip

The case where two extremal layers $y_{extr}^l$ share the same velocity is ruled out by the assumptions. However,
if two different extremal layers have the same velocity, it is expected that the flow is unstable, and in this case $\psi_\alpha(t,y)$ 
is exponentially growing \cite{Bouchet2}.

\medskip

The rest of this paper is organized as follows. In section $3$, we 
 give a complete and new description of solutions of Rayleigh equation both near critical and extremal points, and prove
 Propositions \ref{extension} and \ref{localization}.
We then investigate the long time behaviour of $\psi_\alpha$ and $\omega_\alpha$ in section $4$, 
before studying  the "vorticity depletion property"   in section $5$.


\section{Study of Rayleigh equation \label{studyRay0}}


The aim of this section is to prove  Propositions \ref{extension} and  \ref{localization}.
We first locally construct two independent solutions of Rayleigh equation, then
glue them together to get $\psi_{\pm,\alpha,c}$.

Let us fix some $y_0 \in \rit^+$ and some $c_0 \in \cit$ with $\Im c_0 > 0$. The aim of the  sections $3.1$, $3.2$, $3.3$  and $3.4$ is to describe
locally
the solutions of the homogeneous Rayleigh equation (\ref{Rayleighwithout1}) for $y$ and $c$ in the vicinity of $y_0$ and $c_0$,
by constructing two independent local solutions. We distinguish four cases: away from critical and extremal points, near critical points, near 
extremal points and at infinity.


\subsection{Regular point \label{regularpoint}}


If $y_0$ and $c_0$ are such that $U_s(y_0) \ne c_0$, then locally, 
 Rayleigh equation is  a regular ordinary differential equation.
As a consequence, locally, near $(y_0,c_0)$,  there exist two independent  solutions which depend on a smooth way on $c$ and $y$.


\subsection{Near a critical point}


We consider $y_0$ and $c_0$  such that $U_s(y_0) = c_0$ but $U_s'(y_0) \ne 0$.
If $U_s(y)$ is holomorphic, it is well known \cite{Reid} that there exist two independent solutions $\psi_{1,\alpha}(y,c)$
and $\psi_{2,\alpha}(y,c)$ of
Rayleigh equations which are of the form 
\beq \label{exppsi}
\psi_{l,\alpha}(y,c) = P_{l,\alpha}(y - y_c,c) + (y - y_c) \log(y - y_c) Q_{l,\alpha}(y - y_c,c) 
\eeq
for $l = 1$ and $l = 2$,
where $P_{l,\alpha}$ and $Q_{l,\alpha}$ are holomorphic in both $y$ and $c$, and where $y_c$ is the so called critical layer,
defined by $U_s(y_c) = c$. Note that $y_c$ is then a complex number. This approach requires $U_s$ to be holomorphic and follows methods developed in particular by Fuchs and Frobenius.

In this article we develop a new method of construction, which extends to non holomorphic velocity profiles
$U_s(y)$. We will prove the following Proposition.

\begin{proposition} \label{propfirst2}
If $(y_0,c_0)$ is a critical point, namely if $c_0 = U_s(y_0)$ and $U_s'(y_0) \ne 0$,
then, in a vicinity  of $(y_0,c_0)$, for $\Im c > 0$, there
exist two independent solutions $\psi_{1,\alpha}(y,c)$ and $\psi_{2,\alpha}(y,c)$ of Rayleigh equations,
which are of the form
\beq \label{psifirstcase}
\psi_{l,\alpha}(y,c) =  \psi_{l,\alpha}^r(y,c) + (U_s(y) - c) \log(U_s(y) - c)  \psi_{l,\alpha}^s(y,c) ,
\eeq
where $l = 1$ or $l = 2$, 
and where $\psi_{l,\alpha}^r$ and $\psi_{l,\alpha}^s$ are $C^1$ functions of $y$ and $c$.
Moreover, for any $y \ge 0$, $\psi^r_{l,\alpha}(y,c)$ and $\psi^s_{l,\alpha}(y,c)$  can be extended by continuity on the real line
$\Im c = 0$, to $C^1$ functions of $y$ and $c$. 

Let $R(y_1,y_2,c)$ be the resolvent of Rayleigh equation. Then, $R(y_0 - \eps,y_0 + \eps,c)$ is bounded
provided $\eps > 0 $ is small enough and provided $c$ is close enough to $c_0$.
\end{proposition}

\Remarks
We choose the determination of the logarithm which is defined on $\cit - i \rit^+$, since $\Im c > 0$ and thus 
$U_s(y) - c$ has a negative imaginary part.
We note that $\psi_{1,\alpha,c}$ and $\psi_{2,\alpha,c}$ are defined for $\Im c \ge 0$ only.

\begin{proof}
For such profiles, the extremal layer $y_c$ must be defined slightly differently, namely by 
\beq \label{defextremal}
U_s(y_c) = \Re c .
\eeq
Note that $y_c$ is well defined provided $c$ is small enough, and that
$$
y_c = (U_s'(y_0))^{-1} \Re c + O \Bigl( | \Re c |^2 \Bigr).
$$ 
We first solve the simplified equation
\beq \label{Rayl2}
(U_s - c) \partial_y^2 \psi = U_s'' \psi,
\eeq
namely Rayleigh equation with $\alpha = 0$, whose solutions turn out to be fully explicit.
Namely,
\beq \label{psi1def}
\psi_1^0(y,c) = U_s(y) - c
\eeq
is a particular solution of (\ref{Rayl2}), which is $C^\infty$ in both variables $y$ and $c$, and
 an independent solution, with unit Wronskian, is given by
 \beq \label{Rayl3}
\psi_2^0(y,c) 
= (U_s(y) - c) \Bigl[ \int_{y_1}^y {du \over (U_s(u) - c)^2} + C_1 \Bigr],
\eeq
where $y_1$ and $C_1$ are arbitrary. We have to prove that $\psi_2^0$ is of the form (\ref{psifirstcase}). To study $\psi_2^0$,
we make the change of variables $v = U_s(u)$,  which is locally well defined and leads to 
\beq \label{defiI1}
I_1 := \int_{y_1}^y {du \over (U_s(u) - c)^2} 
= \int_{U_s(y_1)}^{U_s(y)} {dv \over U_s'(U_s^{-1}(v))  (v - c)^2} .
\eeq
We now expand $1/U_s'(U_s^{-1}(v))$ at $\Re c$, which gives
\beq \label{expandI1}
{1 \over U_s'(U_s^{-1}(v))} = {1 \over U_s'(y_c)} + \sum_{k=1}^N \alpha_k(\Re c) \, (v - \Re c)^k +   (v - \Re c)^{N+1} G(v,\Re c),
\eeq
where $\alpha_k(\Re c)$ and  $G(v,\Re c)$ are smooth functions.
We note that the first two terms of (\ref{expandI1}) form a polynomial in $v$. We expand this polynomial at $c$, which leads to
$$
I_1 =
{C_0(c) \over U_s(y) - c} + C_1(c) \log( U_s(y) - c) + H_1(y,c) +  H_2(y,c)
$$
where $C_0$, $C_1$ and $H_1$ are smooth functions and 
$$
H_2(y,c) =  \int_{U_s(y_1)}^{U_s(y)} {( v - \Re c)^{N+1} \over (v - c)^2} G(v,\Re c) \, dv .
$$
Note that through this expansion we reintroduced the imaginary part of $c$ in the formulas.
As $| v - \Re c | / | v - c| \le 1$,
we note that the derivatives of $H_2$ up to the order $N-1$, both in $y$ and $c$, are bounded, thus $\psi_2^0(y,c)$ is of the form
\beq \label{psi201}
\psi_2^0(y,c) = P(y,c) +  (U_s(y) - c) \log (U_s(y) - c) Q(y,c) ,
\eeq
where $P$ and $Q$ are $C^{N-1}$ function which can be extended to $\Im c = 0$ in a $C^{N-1}$ way.
This ends the description of the solutions of Rayleigh equation in the particular case $\alpha = 0$.

\medskip

We  now go back to the full  Rayleigh equation (\ref{Rayleighwithout1}), namely to the case $\alpha \ne 0$. 
Let 
\beq \label{defiG0}
G_0(x,y) = \Bigl\{
\begin{array}{c} \psi_{+,0}(x) \psi_{-,0}(y) \quad \hbox{if} \quad x < y, \cr
\psi_{-,0}(x) \psi_{+,0}(y) \quad \hbox{if} \quad x > y.
\end{array} 
\eeq
Let $\sigma > 0$ and let
\beq \label{defiGreen0}
 {\cal G} f(x) =  \int_{-\sigma}^\sigma G_0(x,y) f(x) \, dx.
\eeq
Then, for $- \sigma \le y \le + \sigma$,
$$
\partial_y^2 {\cal G} f - {U_s'' \over U_s - c} {\cal G} f = f .
$$
We have
$$
 | {\cal G} f(y) | \le |\psi_{-,0}(y)| \int_{-\sigma}^y | \psi_{+,0}(x) | |f(x) | \, dx
+ | \psi_{+,0}(y) | \int_y^{+\sigma} | \psi_{-,0}(x) | | f(x) | \, dx,
$$
thus,
\beq \label{boundit}
\| {\cal G} f \|_{L^\infty} \lesssim  \sigma  \| f \|_{L^\infty}.
\eeq
In particular ${\cal G}$ is a contraction in $L^\infty([-\sigma,+\sigma])$ provided $\sigma$ is small enough.

We now introduce the iterative scheme
\beq \label{iter0}
\psi^\pm_{n+1} = \psi_{\pm,0} + \alpha^2 {\cal G} \psi^\pm_n
\eeq
starting with $\psi^\pm_0 = 0$.
If $\sigma \alpha^2$ is small enough, $\psi_n^\pm(y)$ converge in $L^\infty$ to $\psi_{\pm,\alpha,c}(y)$ which 
satisfy $Ray_{\alpha,c} \psi_{\pm,\alpha,c}(y) = 0$.
Moreover, if $\psi_n$ is of the form (\ref{psifirstcase}), then $\psi_{n+1}$ is of the same form. Thus $\psi_{\pm,\alpha,c}$ is also of the same form.
The bound on the resolvent is straightforward.
\end{proof}


\subsection{Near an extremal point}


We now consider $y_0$ and $c_0$ such $U_s(y_0) = c_0$ and $U_s'(y_0) = 0$, but $U_s''(y_0) \ne 0$.
Note that in this case $y_0 = y_{extr}^l$ and $c_0 = c_{extr}^l$ for some $1 \le l \le L$.
The objective of this section is to study the behaviour of the solutions of Rayleigh equation for $y$ near $y_{extr}^l$ and $c$ near $c_{extr}^l$.
We begin with the explicit computation of the solutions of Rayleigh equation in the particular case $U_s(y) = y^2$ and $\alpha = 0$.


\subsubsection{The case $\alpha = 0$ and $U_s(y) = y^2$ \label{secy2}}


In this particular case, both $\psi_1^0(y,c) = U_s(y) - c$ and $\psi_2^0(y,c)$ are explicit, since, using (\ref{Rayl3}),
after some computations, up to a multiplication by $\sqrt{c}$,  we obtain
\beq \label{psi2explicit}
\psi_2^0(y,c) = - {y \over 2 \sqrt{c}} + {1 \over 4 c} (y^2 - c) \log \Bigl( {y + \sqrt{c} \over y - \sqrt{c}} \Bigr) .
\eeq
Note the "scale" $\sqrt{c}$, which naturally appears near the singularity, the homogeneity in $z = y / \sqrt{c}$, 
and the two neighbouring singularities, at $y = \pm \sqrt{c}$.

Let us first choose the determination of the logarithm which is defined on $\cit - \rit^+$. 
Then, when $y$ is positive and goes to $+\infty$, 
\beq \label{psi21}
\psi_2^0(y,c) = - {1 \over 3 z}  + O \Bigl( {1 \over z^3} \Bigr)
\eeq
and, when $y$ is negative and goes to $- \infty$, 
\beq \label{psi22}
\psi_2^0(y,c) =  i {\pi \over 2} z^2 + O(1).
\eeq
We note that $\psi_3^0(y,c) = \psi_2^0(-y,c)$ is another solution, with a symmetric behaviour, namely
$\psi_3^0(y,c)  \sim i \pi z^2 / 2$ when $y \to + \infty$ and $\psi_3^0(y,c) \sim + 1 / 3  z$ when $y \to - \infty$. Note that
\beq \label{relationpsi23}
\psi_3^0(y,c) = \psi_2^0(y,c) + {i \pi \over 2 c} (y^2 - c).
\eeq
We also note that as $c \to 0$, $c \psi_2^0$ goes to $0$ for positive $y$ and is of order $y^2$ for negative $y$, and symmetrically
for $\psi_3^0$. As $c \to 0$, these two particular solutions $\psi_2^0$ and $\psi_3^0$ get localized on one side of $0$ only.
This "localization property" 
is at the very origin of the "vortex depletion" introduced by F. Bouchet and H. Morita in \cite{Bouchet-Morita-2010}.


\subsubsection{The general case}


In the general case we will prove the following Proposition.

\begin{proposition} \label{propositiondeg}
If $(y_{extr},c_{extr})$ is an extremal point, then, in the vicinity of $(y_{extr},c_{extr})$, for $\Im c > 0$, there
exist two independent solutions $\psi_\pm$ to Rayleigh equation, 
whose Wronskian $W(c)$ is exactly of order $(c - c_{extr})^{-3/2}$, and which can be extended by continuity to $\Im c = 0$, except for $c = c_{extr}$.

When $c$ is real, they satisfy the following estimates.  In a small neighbourhood of $y_{extr}$,
for $y > y_{extr}$,
\beq \label{atplus1}
| \psi_-(y) | \lesssim {1 \over |y - y_{extr}|},  
\eeq
 for $y < y_{extr}$,
\beq \label{atplus2}
| \psi_-(y) | \lesssim { |y - y_{extr}|^2 \over |c - c_{extr}|^{3/2}} , 
\eeq
and when $| y - y_{extr} | \le |c - c_{extr}^l |^{1/2}$,
\beq \label{atplus3}
| \psi_-(y) | \lesssim | c - c_{extr}|^{-1/2}.
\eeq
The estimates on $\psi_+(y)$ are symmetric.

Moreover, the resolvent $R(y_{extr} - \eps, y_{extr} + \eps, c)$ and its inverse are exactly of order $O(|c - c_{extr}|^{-3/2})$,
provided $\eps > 0$ is small enough and provided $c$ is close enough to $c_{extr}$.
\end{proposition}

\begin{proof}
In order to simplify the formulas, 
we change the $y$ variable into $y - y_{extr}$, and the $c$ variable into $c - c_{extr}$ in order to get $y_{extr} = 0$, $c_{extr} = 0$ and
we further rescale $U_s(y)$ such that $U_s''(0) = 2$.
We begin with the particular case $\alpha = 0$.


\subsubsection*{\it Step $1$: Explicit computations in the case $\alpha = 0$}


We first explicitly solve the simplified equation (\ref{Rayl2}), namely Rayleigh equation with $\alpha = 0$,
which leads, as previously, to the explicit smooth solution
\beq \label{psi10def}
\psi_1^0(y,c) = U_s(y) - c,
\eeq
and to another solution $\psi_2^0(y,c)$,  defined by (\ref{Rayl3}), namely by
$$
\psi_2^0(y,c) = (U_s(y) - c) I_1(y)
$$
where
 $$
 I_1(y) = \int_{y_1}^y {du \over (U_s(u) - c)^2} .
 $$
We will choose the lower bound $y_1$ later and we will compute $I_1$ almost explicitly.
 The first step is to replace $U_s(y) - c$ by a polynomial of degree $2$ in the denominator of $I_1$.
 
Let us first assume that $\Re c > 0$. Then, provided $c$ is small enough, there exist
two solutions $y_+$ and $y_-$ to 
$$
U_s(y) = \Re c.
$$ 
Note that, in this definition of the extremal layer, we use the real part of $c$ only, since $U_s$ may not be defined for complex numbers.
The imaginary part of $c$ will be reintroduced later. We have, for some constants $A$ and $B$ which depend on
$U_s'''(0)$ and $U_s''''(0)$,
$$
y_\pm = \pm \sqrt{\Re c} + A \Re c \pm B (\Re c)^{3/2} +  O(\Re c^{2}).
$$ 
We rescale $y$ and $y_\pm$ by introducing
$$
z = {y  \over  \sqrt{\Re c}}, \qquad
z_\pm = {y_\pm \over \sqrt{\Re c} } = \pm 1 + O(\sqrt{\Re c}).
$$
We note that
$$
z_+ + z_- = 2 A \sqrt{\Re c} +  O(({\Re c})^{3/2}), \qquad z_+ - z_- = 2 + O(\Re c).
$$
In particular, 
\beq \label{dede1}
\partial_c(z_+ - z_-) = O(1), \qquad \partial_c^2 (z_+ - z_-) = O(1),
\eeq
whereas 
\beq \label{dede2}
\partial_c z_\pm = O(c^{-1/2}), \qquad\partial_c^2 z_\pm = O(c^{-3/2}).
\eeq
Let
$$
V(u) = {U_s(\sqrt{\Re c}  \, u) - \Re c \over \Re c} .
$$
We note that, locally near $0$, $V(u)$ is monotonic and increasing for $u > 0$ and monotonic and decreasing for $u < 0$. Its minimum,
reached at $0$, is exactly $- 1$, and we have $V(z_\pm) = 0$. In order to replace $V(u)$ by a polynomial of degree $2$,
we introduce $\theta(v)$ defined by
$$
\theta(v) =  \beta (v - z_-) (v - z_+),
$$
where $\beta$ is such that $\min_v \theta(v) = - 1$. The minimum of $\theta$ is reached at $(z_+ + z_-)/2 = O(\sqrt{\Re c})$, thus
\beq \label{defibeta}
\beta = 4^{-1}   (z_+ - z_-)^2 = 1 + O(\Re c).
\eeq
We note that $V(u)$ and $\theta(u)$ have the same zeros and the same minimum.
We now "map" $V(u)$ onto the parabola $\theta(v)$ by making the change of variables $u \to v$ where 
\beq \label{change}
\theta(v) = V(u).
\eeq
More precisely, if $u > 0$ then we define $v$ to be the unique solution of (\ref{change}) which is larger than $(z_+ + z_-)/2$,
and similarly if $u < 0$.  The resolution of (\ref{change}) gives
\beq \label{expression}
v = {z_- + z_+ \over 2} \pm \sqrt{ U_s(\sqrt{\Re c} \, u) \over \beta \Re c } 
\eeq
with a "$+$" if $u > 0$ and a "$-$" if $u < 0$.
This change of variables then gives
$$
I_1(y) = \sqrt{\Re c} \int_{y_1 / \sqrt{\Re c}}^{y / \sqrt{\Re c} } {du \over \Bigl[ \Re c \, V(u)  -i \Im c \Bigr]^2} 
$$
\beq \label{I1bbb}
= \sqrt{\Re c} \int_{v(y_1 / \sqrt{\Re c})}^{v(y/ \sqrt{\Re c})} {G(v) \over \Bigl[ \beta \Re c \,  (v - z_+) (v - z_-) -i \Im c \Bigr]^2} dv
\eeq
where, using (\ref{expression}), 
\beq \label{expressionG}
G(v) =  2 \beta \sqrt{U_s(\sqrt{\Re c} \, u) \over U_s'^2(\sqrt{\Re c} \, u)}  .
\eeq
The denominator of (\ref{I1bbb}) is now exactly a polynomial of degree $2$. 
To go on with the computations, we introduce $z_1$ and $z_2$,  the two (complex) solutions of
$$
\beta \Re c (v - z_- ) (v - z_+) = i \Im c .
$$
We have
\begin{equation*} \begin{split}
\sqrt{\beta \Re c} \, z_{1,2} &= \sqrt{\beta \Re c} {z_- + z_+ \over 2}
\pm {1 \over 2} \sqrt{ \beta \Re c \, (z_- - z_+)^2 + 4 i \Im c}
\\ &=  \pm \sqrt{c} + O(c) .
\end{split} \end{equation*}
If $\Im c = 0$, then $z_{1,2} = z_\pm$.
Note that we have thereby "reintegrated" the imaginary part of $c$ and that
\beq \label{newI1}
I_1 = {1 \over \beta^2 (\Re c)^{3/2}}  \int_{v(y_1 / \sqrt{\Re c})}^{v(y  / \sqrt{\Re c})} {G(v) \over  (v - z_1)^2 (v - z_2)^2 } \, dv.
\eeq
The next step to compute $I_1$ is to approximate $G$ by a polynomial. This will allow to compute the leading orders terms of $I_1$.
We thus approximate this numerator by a polynomial of degree $3$.
Using (\ref{expressionG}), we note that, provided $v$ is bounded, 
$$
G(v) = 1 + O( (\Re c)^{1/2})
$$
and that
\beq \label{higher}
\partial_v^n G(v) = O((\Re c)^{n/2})
\eeq 
for $n = 1$, $2$ and $3$.
Let $G_1$ be the polynomial of degree $3$ which satisfies 
\beq \label{conditions}
G_1(z_{1,2}) = G(z_{1,2}), \qquad
 G_1'(z_{1,2}) = G'(z_{1,2}).
 \eeq
Then, using (\ref{higher}) with $n = 4$, we get
$$
G(v) = G_1(v) + (\Re c)^2 \, (v - z_1)^2 (v - z_2)^2 H(v)
$$
for some smooth function $H$. This leads to $I_1(y) = I_2(y) + I_3(y)$ where
\beq \label{defiI2}
I_2(y) = {1 \over \beta^2 (\Re c)^{3/2}}  \int_{v(y_1 / \sqrt{\Re c})}^{v(y  / \sqrt{\Re c})} {G_1(v) \over  (v - z_1)^2 (v - z_2)^2 } \, dv
\eeq
and 
\beq \label{defiI3}
 I_3(y) = (\Re c)^{1/2}  \int_{v(y_1 / \sqrt{\Re c})}^{v(y/\sqrt{\Re c})}  H(v)  \, dv .
\eeq
If we approximate $G$ by higher order polynomials, we can get a better description of the remainder term $I_3$.

We now explicitly compute $I_2$.
We recall that  $P_0 = 1$, $P_1 = v $, $P_2 = (v - z_1) (v - z_2)$
and $P_3 = v (v - z_1) (v -z_2)$ form a basis of the polynomials of degree $\le 3$. Thus $G_1$ is a combination of
$P_j$ for $0 \le j \le 3$, namely
\beq \label{decompG1}
G_1(v) = \sum_{j = 0}^3 \alpha_j P_j(v),
\eeq
where, taking into account (\ref{higher}) and (\ref{conditions}),
\begin{equation*} \begin{split}
\alpha_0 &=  \alpha^0_0 + \alpha^1_0 c + O(c^{3/2}),
\\ \alpha_1 & = \alpha_1^0 c^{1/2} + \alpha^1_1 c^{3/2} + O(c^2),
 \\ \alpha_2 &= \alpha_2^0 c + O(c^2), 
 \\  \alpha_3 &= \alpha_3^0 c^{3/2} + O(c^2).
\end{split} \end{equation*}
Using (\ref{decompG1}),  $I_2$ is a linear combination of ${\cal J}_j$ for $0 \le j \le 3$, where
\beq \label{defiJj}
{\cal J}_j =   
 \int_{v(y_1 / \sqrt{\Re c})}^{v(y/ \sqrt{\Re c})} {P_j(v) \over  (v - z_1)^2 (v - z_2)^2 } \, dv.
\eeq
We then have
\beq \label{computcalJ}
{\cal J}_j = J_j(v) - J_j(v_1)
\eeq
where $v$ denotes $v(y / \sqrt{\Re c})$ and $v_1$ denotes $v(y_1/ \sqrt{\Re c})$.
We now choose $y_1$ to be a small positive number.

We now explicitly compute the primitives $J_j$ of $P_j(v) / (v - z_1)^2 (v - z_2)^2$, which gives
\begin{equation*} \begin{split}
  J_0(v) &= {2 \over (z_2 - z_1)^3} \log \Bigl( {v - z_1 \over v - z_2} \Bigr) 
- {1 \over (z_2 - z_1)^2} \Bigl[ {1 \over v - z_1} + {1 \over v - z_2} \Bigr],
\\
 J_1(v) &= {z_1 + z_2 \over (z_2 - z_1)^3} \log \Bigl( {v - z_1 \over v - z_2} \Bigr) 
 - {1 \over (z_2 - z_1)^2} \Bigl[ {z_1 \over v - z_1} + {z_2 \over v - z_2} \Bigr],
\\
  J_2(v) &= {1 \over z_1 - z_2} \log \Bigl( { v - z_1 \over v - z_2 } \Bigr), 
\\
J_3(v) &={1 \over z_1 - z_2} \Bigl[ z_1 \log (v - z_1) - z_2 \log(v - z_2) \Bigr].
\end{split} \end{equation*}
 From now on, to simplify the notations, we assume that $c$ is real and define
$$
\sigma = \beta^{2/3} \sqrt{ c} = \sqrt{ c} + O( c^{3/2}).
$$
We then have the following explicit formula for $I_2(y)$
\beq \label{I2expl}
 \sigma^3 I_2(y) =   \alpha_0 {\cal J}_0 +  \alpha_1 {\cal J}_1 +  \alpha_2 {\cal J}_2 +   \alpha_3 {\cal J}_3,
\eeq
where $\alpha_j = O(c^{j/2})$ for $j = 0, 1, 2,$ and $3$.

We now study in detail the various $J_j$. Let $c$ be real to fix the ideas.
Then, for $v > z_1$, we have
 $$
J_0 =  {2 \over (z_2 - z_1)^3} \sum_{n \ge 0} {z_2^n - z_1^n \over n v^n}
- {1 \over (z_2 - z_1)^2} \sum_{n \ge 0} {z_1^{n-1} + z_2^{n-1} \over v^n} . 
$$
Note that this series converges as long as $v > z_1$.
If $n$ is even, then
$$
z_2^n - z_1^n = (z_1 + z_2) n + O( [z_1 + z_2]^3) = O(c^{1/2}),
$$
$$
z_1^{n-1} + z_2^{n-1} = (n-1) (z_1 + z_2) + O( [z_1 + z_2]^3) = O(c^{1/2}),
$$
but on the contrary, if $n$ is odd,
$$
z_2^n - z_1^n = 2 + O( [z_1 + z_2]^2),
$$
and
$$
z_1^{n-1} + z_2^{n-1} = 2 + O( [z_1 + z_2]^2).
$$
Thus, after some computations,
$$
J_0 = \sum_{n \ge 3} a_n v^n,
$$ 
for some constants $a_n$ of the form $a_n^0 + O([z_2 - z_1]^2)$, namely of order $O(1)$ if $n$ is odd, and of the form
$a_n^0 (z_1 + z_2) + O([z_2- z_1]^3)$, namely only of order $O(|c|^{1/2})$  if $n$ is even.
 
 We observe that $J_1$, $J_2$ and $J_3$ have similar properties.
 First $J_1$ starts at $n = 2$, and odd coefficients are of order $O(c^{1/2})$, even coefficients being bounded.
 Moreover, $J_0$ starts at $n = 1$,  odd coefficients being of order $O(1)$ and even coefficients  of order $O(c^{1/2})$.
 The situation is similar for $J_0$ which begins with a $\log v$.
Thus,  using (\ref{I2expl}), 
\beq \label{I2expl2}
\sigma^3 I_2 = b_0 c^{3/2} \log v + b_1 {c \over v} + b_2 {c^{1/2} \over v^2} + {b_3 \over v^3} + b_4 {c^{1/2} \over v^4}  + {b_5 \over v^5}  +  \cdots,
\eeq
where $b_j = b_j^0 + O(c)$.

When $v < z_2$, the expansions of the various $J_j$ are the same, up to  the logarithm which changes by $2 i \pi$.
This leads to another term 
$$
 K = {4 i \pi \alpha_0 \over (z_2 - z_1)^3} + 2 i \pi c^{1/2}  \alpha_1 {z_1 + z_2 \over (z_1 - z_2)^2 }
+ {2 i \pi c \alpha_2  \over z_1 - z_2} + 2 i \pi c^{3/2} \alpha_3 
= {i \pi \over 2} \alpha_0 + O(c).
$$
Moreover, when $c$ is real,
\beq \label{sqrtRec}
v = {z_+ + z_- \over 2} + \sqrt{ U_s(y) \over \beta c} .
\eeq
Let us turn to bounds on $I_3$. The integrand of $I_3$ is bounded, hence $|I_3| \lesssim |y - y_1|$ and similarly, 
$| \partial_y I_3 | \lesssim 1$.

We now construct a first solution $\Psi_-$ defined by
$$
\Psi_-(y) = c^{3/2} (U_s(y) - c)  [ I_2 + I_3].
$$
Then, as $y \to - \infty$
$$
\Psi_-(y) = (U_s(y) - c) \Bigl[ \sigma^{-3/2} K + 
 b_0  \log [c^{1/2} v] + { b_1 \over [c^{1/2} v]} + {b_2  \over [c^{1/2} v]^2} + {b_3 \over [c^{1/2} v]^3}
 $$
 \beq \label{Psiminus}
  + {b_4 c \over [c^{1/2} v]^4}  + {b_5 c \over [c^{1/2} v]^5}  
   + {b_6 c^2 \over [c^{1/2} v]^6}  + {b_7 c^2 \over [c^{1/2} v]^7} 
   +  \cdots + I_3
\Bigr],
\eeq
and as $y \to + \infty$, we have the same expansion, up to the term $\sigma^{-3/2} K$ which is not present.
Note that the $\partial_c \Psi_-$ and $\partial_c^2 \Psi_-$ have similar expressions.

Let us study the various terms of (\ref{Psiminus}). This expression is valid for $v > z_2$, namely for $y > y_+(c)$ which is of order $c^{1/2}$.
Moreover, $c^{1/2} v$ is of order $y$, namely between $O(|c|^{1/2})$ and $O(1)$.
For $y$ of order $O(1)$, all the terms are bounded except the first one, $\sigma^{-3/2} K$, which is large and of order $|c|^{-3/2}$.
For $y$ of order $O(|c|^{1/2})$, the largest term is $b_3 [c^{1/2} v]^{-3}$ which is of order $\sigma^{-3/2}$, that is of the same order
as $\sigma^{-3/2} K$, which allows us to "match" $y > y_+$ with $y < y_-$. 

This leads to, for $y \ge   y_+$,
\beq \label{atplus}
| \Psi_-(y) | \lesssim {1 \over y},
\eeq
and for $y \le   y_-$,
\beq \label{atpluss}
| \Psi_-(y) |  \lesssim |c|^{-3/2} y^2.
\eeq
Similar computations  between $y_-$ and $y_+$ show that $\Psi_-$ is of order $O(|c|^{-1/2})$ in this area.

Now each time we differentiate (\ref{Psiminus}) with respect to $y$, as $|\partial_y [c^{1/2} v] | \lesssim 1$, we loose a factor $y$. This leads to, 
for $y \ge   y_+$,
\beq \label{derivPsiy}
| \partial_y \Psi_-(y) | \lesssim {1 \over y^2},
\eeq
 for $y \le   y_-$,
\beq \label{atpluss}
| \partial_y \Psi_-(y) | \lesssim  |c|^{-3/2} y
\eeq
and for $y_- \le y \le  y_+$,
\beq \label{atplusss}
| \partial_y \Psi_-(y) | \lesssim |c|^{-1}.
\eeq
Moreover, using (\ref{sqrtRec}) and $\partial_c \beta = O(1)$, we have
\beq \label{cv}
c^{1/2} v  = A c + O(c^2) + \beta^{-1/2} \sqrt{U_s(y)},
\eeq
\beq \label{cv1}
\partial_c [c^{1/2} v  ] =  A + O(c) + O(\sqrt{U_s(y)}),
\eeq
and
\beq \label{cv2}
\partial_c^2 [c^{1/2} v ] = O(1) + O(\sqrt{U_s(y)}).
\eeq
We note that, if $|y| \gtrsim |c|^{1/2}$ and $|y| \lesssim 1$, then $[|c|^{1/2} v]^{-1} \lesssim |c|^{-1}$ and $[|c|^{1/2} v]^{-1} \lesssim |y|^{-1}$.
Hence, as $\partial_c [c^{1/2} v]$ is bounded by $O(1 + |y|)$, one derivative with respect to $c$ leads to the loss
of a factor $|c|^{-1}$ or $|y|$.

For $y \ge  y_+$, we have
\beq \label{atplussd1}
| \partial_c \Psi_-(y) | \lesssim {1 \over y [c^{1/2} v]} \lesssim \min \Bigl( {1 \over y |c|},  {1 \over |y|^2} \Bigr),
\eeq
\beq \label{atplussd2}
| \partial_c^2 \Psi_-(y) | \lesssim {1 \over y [c^{1/2} v^2]}   \lesssim \min \Bigl( {1 \over y |c|^2},  {1 \over |y|^3} \Bigr) ,
\eeq
and for $y <  y_-$,
\beq \label{atplussd3}
| \partial_c \Psi_-(y) | \lesssim |c|^{-3/2} y , \qquad
| \partial_c^2 \Psi_-(y) | \lesssim |c|^{-3/2}  .
\eeq
Moreover, for $ y_- \le y \le  y_+$,
\beq \label{atplussd4}
| \partial_c \Psi_-(y) | \lesssim |c|^{-3/2} , \qquad
| \partial_c^2 \Psi_-(y) | \lesssim |c|^{-5/2}  .
\eeq
In the case $\Re c < 0$, we introduce $\theta(v)$ defined by $\theta(v) = - \Re c + v^2$
and make the change of variables $u \to v$ where $\theta(v) = U_s(x) - c$.
The computations are then similar to the case $\Re c > 0$.

We construct $\Psi_+$ in a symmetric way.
Direct computations show that the Wronskian of $\Psi_-$ and $\Psi_+$ and the resolvent $R(-\eps,\eps,c)$ are   exactly of order $O(|c|^{-3/2})$.
 This ends the study of the case $\alpha = 0$.


\subsubsection*{\it Step $2$:  The case $\alpha \ne 0$, near $0$}


We  now go back to the full  Rayleigh equation (\ref{Rayleighwithout1}), namely to the case $\alpha \ne 0$. 
Near the extremum, $\alpha^2$ is negligible with respect  $U_s'' / (U_s - c)$, thus we follow a perturbative approach.
We recall that the Wronskian between $\Psi_-$ and $\Psi_+$ is of order $O(c^{-3/2})$. Let us introduce $G_1$ the Green function of 
$$
\partial_y^2 \phi =   {U_s'' \over U_s - c} \phi + \delta_x
$$
which is explicitly given by 
 \beq \label{defiG1}
G_1(x,y) = W^{-1} \Bigl\{
\begin{array}{c} \Psi_+(x) \Psi_-(y) \quad \hbox{if} \quad x < y, \cr
\Psi_-(x) \Psi_+(y) \quad \hbox{if} \quad x > y.
\end{array} 
\eeq
Let $\sigma > 0$ and let
\beq \label{defiGreen1}
 {\cal G} f(x) =  \int_{-\sigma}^\sigma G_1(x,y) f(x) \, dx.
\eeq
Let us first bound ${\cal G}$. 
Let us define the norm $\| \phi \|_\Psi$ by
$$
\| \phi \|_\Psi = \sup_{-\sigma \le x \le \sigma} | \Psi_-(x)|^{-1}  \, | \phi(x)| .
$$
We have
$$
{\cal G}f(y)  
  =  \Psi_-(y)  \int_{-\sigma}^y  { \Psi_+(x)  \over W} f(x)  \, dx,
+  { \Psi_+(y) \over W}  \int_y^{+\sigma}  \Psi_-(x)   f(x)  \, dx.
$$
 As $W^{-1} \Psi_+$ is bounded, the first integral $I_1$ is bounded by
$$
| I_1 | \lesssim \sigma | \Psi_-(y) | .
$$
Let us turn to the second integral $I_2$.
For $y > y_+$, $| \Psi_-(y)|^2 \lesssim |y|^{-2}$, thus the integral is bounded by $|y|^{-1}$ and, as $|W^{-1} \Psi_+| \lesssim \sigma^2$, 
we have $| I_2| \lesssim \sigma^2 | \Psi_-(y)|$. The other cases are similar. 

Thus, for $|c|$ small enough,
\beq \label{eestG}
  \| {\cal G} f(y) \|_\Psi \lesssim  \sigma \| f \|_\Psi
\eeq
and ${\cal G}$ is  a contraction in this space provided $\sigma$ is small enough.

We  define a first solution $\psi_-$ of Rayleigh equation with $\alpha \ne 0$  through the iteration
\beq \label{iter00}
 \phi_{n+1} = \Psi_- + \alpha^2 {\cal G}  \phi_n,
\eeq
starting with $\phi_0 = \Psi_-$.
Then   $\phi_n$ converges in the norm $\| \cdot \|_\Psi$ to a solution
 $\phi_-$ of $Ray_{\alpha,c} \, \phi_- = 0$, which in fact has the same bounds as $\Psi_-$.
 Symmetrically, we construct another independent solution $\phi_+$ which satisfies symmetric bounds.

Estimates on first and second derivatives with respect to $c$ may then be obtained by differentiating the iteration scheme (\ref{iter00}).
More precisely, we have
$$
\partial_c \phi_{n+1} = \partial_c \psi_1^0 + \alpha^2 {\cal G} \partial_c \phi_n + \alpha^2 (\partial_c {\cal G}) \phi_n
$$
and
$$
\partial_c^2 \phi_{n+1} = \partial_c^2 \psi_1^0 + \alpha^2 {\cal G} \partial_c^2 \phi_n + 2  \alpha^2 (\partial_c {\cal G})\partial_c \phi_n
+ \alpha^2 ( \partial_c^2 {\cal G}) \phi_n,
$$
where $\partial_c {\cal G}$ and $\partial_c^2 {\cal G}$ have kernels $\partial_c G_0$ and $\partial_c^2 G_0$. 
According to (\ref{atplussd1}) to (\ref{atplussd4}), each derivative in $c$ leads to a loss of a factor $c$, thus the estimates are similar, and
(\ref{atplussd1}) and (\ref{atplussd4}) hold true for $\psi_-$.
In a symmetric way, we construct an independent solution $\psi_+$.

Let us turn to the estimate on the resolvent. The Wronskian matrix is
$$
W(y) = \Bigl( \begin{array}{cc} 
\psi_-(y) & \psi_+(y) \cr
\partial_y \psi_-(y) & \partial_y \psi_+(y) \cr
\end{array} \Bigr) 
$$
and the resolvent is
$$
R(y_1,y_2) = W(y_2) W^{-1}(y_1) .
$$
For $y < 0$,
$$
W(y) = \Bigl( \begin{array}{cc}
O(c^{-3/2}) & O(1) \cr
O(c^{-3/2}) & O(1) \cr \end{array} \Bigr) 
$$
and conversely for $y > 0$. Thus, if $y_1 > 0 > y_2$,
$R(y_1,y_2) = O(c^{-3/2})$, and is in fact exactly of that order, and similarly for $y_1 < 0 < y_2$.
 \end{proof}


\subsection{At infinity}


It remains to handle the case where $y$ is large.
We distinguish two cases.
If $c_0 \ne U_+$, then, for large $y$, $U_s(y) - c$ does not vanish. Rayleigh equation is regular at infinity and there exist
two solutions  $\psi_{1,\alpha,c}$ and $\psi_{2,\alpha,c}$, 
with unit Wronskian, which behave like $e^{\pm \alpha y}$ for $y \ge A$ with $A$ large enough.

 If on the contrary $c_0 = U_+$,
 then $U_s(y) - c_0 \sim C_+ e^{- \beta y}$ at infinity for some positive constant $\beta$, hence, for this particular value of $c_0$,
$\psi_{1,\alpha,c}$ and $\psi_{2,\alpha,c}$ asymptotically behave like $e^{\pm (\alpha^2 + \beta^2) y}$. 
However, standard arguments give the existence of a smooth solution $\psi_{1,\alpha,c}$ which goes to $0$ exponentially
fast at infinity. We define $\psi_{2,\alpha,c}$ through the Wronskian.


\subsection{Proof of Propositions \ref{extension} and \ref{localization} \label{construction}}


In this section we prove Propositions \ref{extension} and \ref{localization}.

\medskip

{\bf Proof of Proposition \ref{extension}.}
For each couple $(y_0,c_0)$, including $y_0 = +\infty$, we can  locally construct  two solutions $\psi_{1,\alpha,c}$ and
$\psi_{2,\alpha,c}$ of Rayleigh equation. 

By compactness, we can thus pave $\rit^+ \times \{ c, | \Re c | \le A, | \Im c | \le A\}$
for any arbitrarily large positive number $A$, by a finite number of tiles of the form $I(y_p) \times J(c_q)$ 
for some positive numbers $(y_p)_{1 \le p \le P}$ and complex numbers $(c_q)_{1 \le q \le Q}$,
where $I(y_p)$ is either $[y_p - \sigma, y_p + \sigma]$ or $[y_p - \sigma, + \infty)$,
 and $J(c_q)$ is the ball of centre $c_q$ and radius $\sigma$,
provided $\sigma$ is small enough.

We can further impose that some of these $c_q$ are real. They are then called $c_r^0$, and we can enforce that
$[-A - i \eps, A+ i \eps ] \subset \cup_r J(c_r^0)$ provided $\eps$ is small enough. We can moreover enforce that either
$c_r^0$ is a extremal value or $J(c_r^0)$ contains no extremal value.

On each $I(y_p) \times J(c_q)$, we have two independent solutions of Rayleigh equations, constructed 
by Propositions \ref{propfirst2} or \ref{propositiondeg} .

Let us now construct two independent global solution $\psi_{\pm,\alpha,c}$.
Let us fix some $\alpha > 0$. For $y$ large enough, we define $\psi_{-,\alpha,c}$ to be the solution
which goes to $0$ at infinity and is equivalent to $e^{- |\alpha| y}$ (if $c \ne U_+$).
We then extend this solution towards $y = 0$, using either the construction near a regular, critical or extremal
 point. This gives a first global solution, defined on $\rit^+$. Using Proposition $3.1$ and $3.2$,
the construction can be done even for $\Re c = 0$, except  the extremal velocities, which proves the proposition.

\medskip

{\bf Proof of Proposition \ref{localization}.}
We define another solution, called $\psi_{+,\alpha,c}$, using the Wronskian, through
\beq \label{defiW}
\psi_{+,\alpha,c}(y) =  \psi_{-,\alpha,c}(y) \int_y^{+\infty} {1 \over \psi_{-,\alpha,c} ^2(x) } dx .
\eeq
We note that $|\psi_{+,\alpha,c}(y)|$ goes to $+\infty$ as $y \to + \infty$
and behaves like $e^{+ | \alpha | y}$  (if $c \ne U_+$).

If $c$ is not  close to an extremal velocity, then,
 in this process, we never go close to an extremal point $(y_{extr}^l,c_{extr}^l)$. All the resolvents are thus bounded,
and the corresponding solutions $\psi_{-,\alpha,c}$ and $\psi_{+,\alpha,c}$ are bounded, locally in $c$ and uniformly in $y$ for $\psi_{-,\alpha,c}$,
and locally in $c$ and  $y$ for $\psi_{+,\alpha,c}$.

If, on the contrary, $c$ is close to some extremal $c_{extr}^l$, then we must go near $(y_{extr}^l,c_{extr}^l)$, where the resolvent
$R(y_{extr}^l + \sigma, y_{extr}^l - \sigma,c)$ is only bounded by
$(c - c_{extr}^l)^{-3/2}$. As a consequence, $\psi_{-,\alpha,c}(y)$ may not be bounded uniformly in $c$ for $y < y_{extr}$.
We only have $| \psi_{-,\alpha,c}(y) | \lesssim | c - c_{extr}^l |^{-3/2}$.
The study of the magnitude of $\psi_{-,\alpha,c}(y)$ for $y < y_{extr}^l$ is a key point in the study of the localization property
and of the vorticity depletion property.

Let $\sigma$ be small. Then on $[y_{extr}^l - \sigma, y_{extr}^l + \sigma]$, two independent solutions $\psi_\pm$ are described
in Proposition \ref{propositiondeg}, with the Wronskian $W_1$ of order $(c - c_{extr}^l)^{-3/2}$.
We decompose $\psi_{\pm,\alpha,c}$ on this basis 
\beq \label{decomptilde}
 \psi_{\pm,\alpha,c}(y) = a_\pm^1 \psi_-(y) + a_\pm^2 \psi_+(y).
\eeq
At $y = y_{extr}^l + \sigma$, this gives
\begin{equation} \begin{split} \label{transfert}
\Bigl( \begin{array}{c} 
a_\pm^1 \cr a_\pm^2 \cr \end{array} \Bigr) 
&= {1 \over W_1}  \Bigl( \begin{array}{cc}
\psi_{+}'(y+\sigma) & - \psi_{+}(y+\sigma) \cr
- \psi_{-}'(y+\sigma) & \psi_{-}(y+\sigma) \cr \end{array} \Bigr)
\Bigl( \begin{array}{c} 
\psi_{\pm,\alpha,c} (y+\sigma) \cr \psi_{\pm,\alpha,c}' (y+\sigma) \cr \end{array} \Bigr) 
\\ &=   \Bigl( \begin{array}{cc}
O ( 1)  &  O ( 1)  \cr
 O (  | c - c_{extr}^l |^{3/2}) & O (  | c - c_{extr}^l |^{3/2}) \cr \end{array} \Bigr)
\Bigl( \begin{array}{c} 
\psi_{\pm,\alpha,c} (y+\sigma) \cr \psi_{\pm,\alpha,c}' (y+\sigma) \cr \end{array} \Bigr) .
\end{split} \end{equation}
Thus
\beq \label{firstestimate}
a_\pm^1 = O(1) , \qquad
a_\pm^2 = O( | c - c_{extr}^l |^{3/2}) .
\eeq
Two cases arise. Case I: at $y + \sigma$, 
$$
 a_-^1 =  \psi_+' \psi_{-,\alpha,c} - \psi_+ \psi_{-,\alpha,c}' 
$$
 remains bounded away from $0$ as $c \to c_{extr}^l$. Then in view of (\ref{decomptilde}),
$\psi_{-,\alpha,c}(y_{extr}^l- \sigma)$ and $\psi_{-,\alpha,c}'(y_{extr}^l- \sigma)$  are exactly of order $O( | c - c_{extr}^l|^{-3/2})$ and thus $\psi_{-,\alpha,c}(y)$
is of the same order  for $y < y_{extr}^l - \sigma$ (since the resolvent is bounded for $y < y_{extr}^l - \sigma$).
In view of (\ref{defiW}), $\psi_{+,\alpha,c}(y)$ and $\psi_{+,\alpha,c}'(y)$ are then of order $O(|c - c_{extr}|^{3/2})$ for $y < y_{extr}^l - \sigma$.

Case II: on the contrary, there exists a sequence of complex numbers $c_n$ such that $c_n \to c_{extr}^l$ and
such that $a_-^1$ goes to $0$ as $c \to c_{extr}^l$. 
Then, as the Wronskian of $\psi_{\pm,\alpha,c}$ equals $1$, 
$(\psi_{-,\alpha,c_n},\psi_{-,\alpha,c_n}')$ and $(\psi_{+,\alpha,c_n},\psi_{+,\alpha,c_n}')$  are not colinear. 
Thus $a_+^1$ does not go to $0$.
In particular, $\psi_{+,\alpha,c_n}(y_{extr}^l - \sigma)$ and $\psi_{+,\alpha,c_n}'(y_{extr}^l - \sigma)$  are exactly of order $O( | c_n - c_{extr}^l|^{-3/2})$, and thus
$\psi_{+,\alpha,c_n}(y)$ is of the same order for $y < y_{extr}^l - \sigma$.
In view of (\ref{defiW}), this implies that $\psi_{-,\alpha,c_n}(y_{extr}^l - \sigma)$, together with $\psi_{-,\alpha,c_n}'(y_{extr}^l - \sigma)$, are of order
$O(| c - c_{extr}^l |^{3/2})$. Thus  $\psi_{-,\alpha,c}(0)$
goes to $0$ as $c_n \to c_{extr}^l$, which is contradictory with the
assumption (A3).

Therefore $\psi_{-,\alpha,c}(y)$ is of order $| c - c_{extr}^l |^{-3/2}$ if $y < y_{extr}^l - \sigma$ and $\psi_{+,\alpha,c}(y)$ is bounded 
by $|c - c_{extr}|^{3/2}$.

Summing up the estimates, we have
$$
(\psi_{-,\alpha,c}(y-\sigma), \psi_{-,\alpha,c}'(y - \sigma)) = O(|c - c_{extr}^l |^{-3/2}),
$$
$$
(\psi_{+,\alpha,c}(y-\sigma), \psi_{+,\alpha,c}'(y - \sigma)) = O(|c - c_{extr}^l |^{+3/2}),
$$
together with
\beq \label{decompo}
a_-^1 = O(1), \qquad
a_-^2 = O(| c - c_{extr} |^{3/2})
\eeq
and 
\beq \label{decompo2}
a_+^1 = O( |c - c_{extr} |^3), \qquad
a_+^2 = O( | c - c_{extr} |^{3/2} ).
\eeq
This leads to (\ref{bbbb1}), (\ref{bbbb2}) and (\ref{bbbb3}).
Note that (\ref{decompo}) and (\ref{decompo2}), together with the estimates of Proposition \ref{propositiondeg}
give a very accurate description of $\psi_{\pm,\alpha,c}$ near the extremal velocity. More precisely, near an extremal point $y_{crit}^l$,
for $c$ close to $c_{crit}^l$, and for $y > y_+(c)$,
\beq \label{1}
| \psi_-(y) | \lesssim {1 \over | y - y_{extr}^l |}, \qquad
| \psi_+(y) | \lesssim | y - y_{extr}^l |^2,
\eeq
for $y < y_-(c)$,
\beq \label{2}
| \psi_-(y) | \lesssim { | y - y_{extr}^l |^2 \over  | c - c_{extr}^l |^{3/2} }, \qquad
| \psi_+(y) | \lesssim  { | c - c_{extr}^l |^{3/2} \over| y - y_{extr}^l |},
\eeq
and for $|y - y_{extr}^l|$ of order $O(|c|^{1/2})$,
\beq \label{3}
| \psi_-(y) | \lesssim   | c - c_{extr}^l |^{-1/2} , \qquad
| \psi_+(y) | \lesssim   | c - c_{extr}^l |.
\eeq
Note that the estimates on $\psi_+$ and $\psi_-$ are not symmetric since we enforce that both of them are of order $O(1)$ for large $y$.
This finishes the proof of the proposition.


\section{Asymptotic behavior of $\psi_\alpha(t,y)$ }


The aim of this section is to obtain pointwise bounds on $\psi_\alpha$ and $\partial_y \psi_\alpha$, 
namely to prove (\ref{correctebound1}) and  (\ref{correctedbound2}).
We first describe the Green function of Rayleigh equation and decompose its solution $\psi_{\alpha,c}(t,x)$
in an "interior part" $\psi_{\alpha,c}^{int}(t,x)$ and a "boundary part" $\psi_{\alpha,c}^b(t,x)$ that we then study in details.
The boundary part $\psi_{\alpha,c}^b(t,x)$ will itself be decomposed in the  projection   on eigenmodes  $\psi_{\alpha,c}^{b,modes}(t,y)$
and   a decaying remainder $\psi_{\alpha,c}^{b,decay}(t,y)$.


\subsection{Green function of Rayleigh equation \label{Gb}}


We define the Green function $G_{\alpha,c}(x,y)$ of Rayleigh equation to be the solution of
\beq \label{defiG0c}
Ray_{\alpha,c}(G_{\alpha,c}(x,y)) = \delta_x
\eeq
which satisfies $G_{\alpha,c}(x,0) = 0$ and $G_{\alpha,c}(x,y) \to 0$ as $y \to + \infty$. 
The solution of Rayleigh equation
\beq \label{Ray0cf}
 (U_s - c)  (\partial_y^2 - \alpha^2 ) \psi_{\alpha,c} - U_s''  \psi_{\alpha,c}  = f,
 \eeq
 together with its two boundary conditions is then explicitly given by
 \beq \label{constructionpsi}
 \psi_{\alpha,c}(y) = \int_0^{+\infty} G_{\alpha,c}(x,y) f(x) dx ,
 \eeq
 and the Green function of Euler equations is the inverse Laplace transform of the Green function of Rayleigh equation, namely
\beq \label{G11}
G_\alpha(t,x,y) =  - { \alpha \over 2 \pi} \int_\Gamma e^{- i \alpha c t } G_{\alpha,c}(x,y) \, dc ,
\eeq
where $\Gamma$ is a contour "above the eigenvalues".

 To construct $G_{\alpha,c}$, we first solve (\ref{defiG0c}) without the boundary condition at $y = 0$, just choosing a particular solution
 which goes to $0$ as $y \to +\infty$. This leads to the introduction of
 $$
 G_{\alpha,c}^{int}(x,y) = {1 \over U_s(x) - c} 
 \Bigl\{ \begin{array}{c}
 \psi_{-,\alpha,c}(y) \psi_{+,\alpha,c}(x) \qquad \hbox{if} \qquad y > x \cr
 \psi_{-,\alpha,c}(x) \psi_{+,\alpha,c}(y) \qquad \hbox{if} \qquad y < x , \cr
\end{array} 
 $$
where we recall that  the Wronskian $W[\psi_{+,\alpha,c},\psi_{-,\alpha,c}] = 1$.
 We then correct $G_{\alpha,c}^{int}$ by $G_{\alpha,c}^b$, such that 
 $G_{\alpha,c} = G_{\alpha,c}^{int} + G_{\alpha,c}^b$
 satisfies both Rayleigh equation and the two boundary conditions. This leads to 
\beq \label{expressGb}
 G_{\alpha,c}^b(x,y) = -  { \psi_{+,\alpha,c}(0)  \over \psi_{-,\alpha,c}(0) } 
 {\psi_{-,\alpha,c}(x) \over U_s(x) - c}  \psi_{-,\alpha,c}(y)  .
\eeq
We note that $G_{\alpha,c}^{int}(x,y)$ is smooth for $c$ outside the range of $U_s(y)$
and that the dispersion relation $\psi_{-,\alpha,c}(0) = 0$ only appears in $G_{\alpha,c}^b(x,y)$. In particular,
$G_{\alpha,c}^b(x,y)$ has poles at the various eigenvalues $c \in \sigma_P$.


\subsection{Study of $G_\alpha^{int}$}


Let $x > 0$ and $y > 0$ be fixed.
By construction, $x \in I(x_p)$ for some $1 \le p \le P$ and $y \in I(y_{p'})$ for some $1 \le p' \le P$.
We recall that, for $x < y$,
\beq \label{definitionGalpha}
G_{\alpha}^{int}(t,x,y) = \int_\Gamma {  \psi_{+,\alpha,c}(x)   \psi_{-,\alpha,c}(y) \over  U_s(x) - c} e^{- i \alpha c t} \, dc ,
\eeq
where $\Gamma$ is some contour "above" the spectrum.
The expression for $y < x$ is symmetric.
We define $\psi_\alpha^{int}(t,y)$ to be
\beq \label{defipsialphaint}
\psi_\alpha^{int}(t,y) = \int_{\rit^+} G_\alpha^{int}(t,x,y) \omega_\alpha^0 (x) \, dx .
\eeq
As the integrand of $G_{\alpha}^{int}$ has no pole if $\Im c > 0$, we can move $\Gamma$ to
$$
\Gamma_\eps = \Gamma_1(\eps) \cup \Gamma_2(\eps) \cup \Gamma_3(\eps)
$$
$$
= \{ - A + i \eps + (1+ i) \rit^- \} \cup [-A + i \eps, + A + i \eps ] \cup \{ A + i \eps + (1 - i) \rit^+ \} ,
$$
where $\eps$ is a small positive number and $A$ a large real number so that the range of $U_s(y)$ is included in
$[-A/2,+A/2]$.
We focus on the integral over $\Gamma_2(\eps)$, the other ones being similar.

Let $\chi_r$ be a partition of unity adapted to the $c_r^0$, namely a set of smooth functions such that
$$
\sum_r \chi_r (y) = 1
$$
for all $y$ with $|\Re y| \le A$ and $-\eps \le \Im y \le \eps$ provided $\eps$ is small enough, 
and such that the support of $\chi_r$ is included in $J(c_r^0)$. We can choose the partition of the unity, such that,
on the support of $\chi_r$, there is at most one extremal point, which in this case is $c_r^0$.

We define, omitting $\eps$ in the notation,
\beq \label{definitionGalphat}
G_{\alpha}^{int,r}(t,x,y) = \int_{\Gamma_\eps}
\chi_r(c)  {  \psi_{+,\alpha,c}(x)   \psi_{-,\alpha,c}(y) \over U_s(x) - c} e^{- i \alpha c t} \, dc 
\eeq
and
\beq \label{defipsiintr}
\psi_\alpha^{int,r}(t,y) = \int_{x \ge 0} G_{\alpha}^{int,r}(t,x,y) \omega_\alpha(x) \, dx.
\eeq
We have
$$
\psi_\alpha^{int} (t,y) = \sum_r \psi_{\alpha}^{int,r}(t,y) .
$$
The study of $G_\alpha^{int,r}$ and $\psi_\alpha^{int,r}$ depends on whether $c_r^0$ is an extremal velocity or not.


\subsubsection{Away from extremal values \label{awayfrom}}


\begin{lemma} \label{decay21}
If $c_r^0$ is not an extremal velocity then
\beq \label{decay2}
| \psi_\alpha^{int,r}(t,y) | \lesssim \langle \alpha t \rangle^{-2} 
\eeq
and
\beq \label{decay2der}
| \partial_y \psi_\alpha^{int,r}(t,y) | \lesssim \langle \alpha t \rangle^{-1} 
\eeq
\end{lemma}

\Remark This result is classical \cite{Schmidt}. We only sketch its proof.

\begin{proof}
Let us assume that $x < y$ to fix the ideas.
Away from the extremal velocities,  near $x$ (respectively near $y$), section \ref{regularpoint} 
and Proposition \ref{propfirst2} give two independent functions $\psi^x_{1,\alpha,c}$ and $\psi^x_{2,\alpha,c}$ (respectively
$\psi^y_{1,\alpha,c}$ and $\psi^y_{2,\alpha,c}$)
with unit Wronskian, defined near $x$ (resp. $y$).  The first step is to decompose $\psi_{+,\alpha,c}(x)$ on these two particular solutions
and to write
\beq \label{decoco}
\psi_{+,\alpha,c}(x) = a_+^1 \psi^x_{1,\alpha,c}(x) + a_+^2 \psi^x_{2,\alpha,c}(x),
\eeq
and we can decompose  $\psi_{-,\alpha,c}$ similarly,
where $a_\pm^{1,2}$ are some constants.
As all the resolvents used in the construction of $\psi_{\pm,\alpha}$ are bounded, $a_\pm^{1,2}$ are bounded.
 Up to a multiplication by $a_-^m a_+^n$ with $m$, $n = 1$ or $2$ (which are bounded),
we are then  led to study
\beq \label{Gmn}
G_{m,n}^r(t,x,y) =  \int_{\Gamma_\eps} \chi_r(c) {  \psi_{m,\alpha,c}(x)   \psi_{n,\alpha,c}(y) \over  U_s(x) - c} e^{- i \alpha c t} \, dc 
\eeq
where $m$ and $n$ equal $1$ or $2$. To alleviate the notation, we have dropped the superscripts $x$ and $y$.

If $c \ne U_s(x)$ and $c \ne U_s(y)$ on the support of $\chi_r(c)$, the integrand of $G_{m,n}^r$ is smooth.
Integrating by parts, we get $| G_{m,n}^r(t,x,y) | \lesssim \langle \alpha t \rangle^{-2}$, which leads to the desired decay
$| \psi_\alpha^{int,r}(t,y) | \lesssim  \langle \alpha t \rangle^{-2}$.

If $c = U_s(x)$ or $c = U_s(y)$ on the support of $\chi_r(c)$, then the integrand of (\ref{Gmn})  is singular.
 We will discuss the case where both $U_s(x)$ and $U_s(y)$ are in the support of $\chi_r$, the
other cases being similar.
In this case, the integrand of $G_{m,n}^r$ is no longer smooth, but
 has "$z \log z$" branches at $c = U_s(x)$ and $c = U_s(y)$, and a pole at $c = U_s(x)$. 
However, it turns out that, after a few computations that we will now detail, $G_{m,n}^r$ is almost explicit. 
The following computations are very classical \cite{Schmidt}.

We recall that $\psi_{m,\alpha,c}(x)$ and $\psi_{m,\alpha,c}(y)$ 
are of the form (\ref{psifirstcase}), thus $G_{m,n}^r(t,x,y)$ is the sum of four terms $K_1$-$ K_4$.
To ease reading, we put $c$ into the variables instead of into the subscript.
The first term is
$$
K_1 =\lim_{\eps \to 0^+} \int_{\Gamma_\eps} \chi_r(c) e^{- i \alpha c t }  {\psi_{m,\alpha}^r(x,c) \psi_{n,\alpha}^r(y,c) \over U_s(x) - c} \; dc 
$$
\beq \label{contrib1}
= -  2 i \pi \psi_{m,\alpha}^r(x,U_s(x)) \psi_{n,\alpha}^r(y,U_s(x)) e^{- i \alpha U_s(x) t}  \chi_r(U_s(x)) + O(\langle \alpha t \rangle^{-2}).
\eeq
(see the Appendix for more details on this computation and the  following ones).
The second term is
$$
K_2 = \lim_{\eps \to 0^+} \int_{\Gamma_\eps} \chi_r(c) e^{- i \alpha c t }  {(U_s(x) - c) \log (U_s(x) -c) 
\psi_{m,\alpha}^s(x,c) \psi_{n,\alpha}^r(y,c) \over U_s(x) - c} \; dc 
$$
\beq \label{contrib2}
= -2  i \pi \psi_{m,\alpha}^s(x,U_s(x)) \psi_{n,\alpha}^r(y,U_s(x))  {e^{- i \alpha U_s(x) t} \over \alpha t}  \chi_r(U_s(x))
+   O(\langle \alpha t \rangle^{-2}) .
\eeq
The third term of $G_{m,n}^r(t,x,y)$  is 
$$
K_3 =  \lim_{\eps \to 0^+} \int_{\Gamma_\eps} \chi_r(c) e^{- i \alpha c t }  { \psi_{m,\alpha}^r(x,c) (U_s(y) - c) \log (U_s(y) -c) 
 \psi_{n,\alpha}^s(y,c) \over U_s(x) - c} \; dc 
$$
$$
= -  2 i \pi \psi_{m,\alpha}^r(x,U_s(x)) \psi_{n,\alpha}^s(y,U_s(x)) [U_s(y) - U_s(x)]
$$
\beq \label{contrib3}
\times \log (U_s(y) - U_s(x)) e^{- i \alpha U_s(x) t}  \chi_r(U_s(x))
+ Q_s(t,x,y)
+   O(\langle \alpha t \rangle^{-3}),
\eeq
where
$$
Q_s(t,x,y) =  \pi {\psi^r_{m,\alpha}(x,U_s(y))
 \psi^s_{n,\alpha,}(y,U_s(y))  \over U_s(x) - U_s(y)}  {e^{- \alpha i U_s(y) t} \over \alpha^2 t^2} \chi_r(U_s(y)).
$$
We note that $Q_s(t,x,y)$ is singular when $x = y$.
For $x$ close to $y$, we must handle the third term of $G_{m,n}^r(t,x,y)$ differently. We expand $\psi_{n,\alpha}^s(y,c)$
and $\psi_{n,\alpha}^r(x,c)$ as
$$
\psi_{n,\alpha}^s(y,c) = \psi_{n,\alpha}^s(y,U_s(x)) + (c - U_s(x)) \widetilde \psi_{n,\alpha}^s(y,c)
$$
and
$$
\psi_{n,\alpha}^r(x,c) = \psi_{n,\alpha}^r(x,U_s(x)) + (c - U_s(x)) \widetilde \psi_{n,\alpha}^r(x,c)
$$
where $\widetilde \psi_{n,\alpha}^s$ and  $\widetilde \psi_{n,\alpha}^r$ are smooth functions. 
Using these decompositions, $K_3$ is a sum of four terms. The most difficult term is 
$$
\psi_{n,\alpha}^s \Bigl( y,U_s(x) \Bigr) \psi_{n,\alpha}^r \Bigl( x,U_s(x) \Bigr) L(t,x,y)
$$
where
 $$
 L(t,x,y) = \lim_{\eps \to 0^+} \int_{\Gamma_\eps} \chi_r(c) e^{- i \alpha c t }  { (U_s(y) - c) \log (U_s(y) -c) 
\over U_s(x) - c} \; dc.
$$
Using $U_s(y) - c = ( U_s(y) - U_s(x)) + (U_s(x) - c)$, we get
$$
L(t,x,y)  = L_1(t,x,y) + L_2(t,x,y),
$$
where
$$
L_1(t,x,y) = - 2 i \pi (U_s(y) - U_s(x)) \log(U_s(y) - U_s(x))  e^{- i \alpha U_s(x) t}
$$
and $L_2$ is smooth, oscillatory and decays like $\langle \alpha t \rangle^{-1}$.

The last term of $G_{m,n}^r(t,x,y)$ is
$$
K_4 = \lim_{\eps \to 0^+} \int_{\Gamma_\eps} \chi_r(c) e^{- i \alpha c t }   \log (U_s(x) - c) \psi_{m,\alpha}^s(x,c) (U_s(y) - c) \log (U_s(y) -c) 
 \psi_{n,\alpha}^s(y,c)  \; dc 
$$
\beq \label{contrib4}
= - { \pi \over \alpha t} \psi^s_{m,\alpha}(x,U_s(x))\psi^s_{n,\alpha}(y, U_s(x)) (U_s(y) - U_s(x))   \chi_r(U_s(x))
\eeq
$$
\times  \log( U_s(y) - U_s(x)) e^{-i \alpha U_s(x) t}  + O(\langle \alpha t \rangle^{-2}).
$$
The limiting Green function $\lim_{\eps \to 0^+} G_{m,n}^r(t,x,y)$ is the sum of (\ref{contrib1}), (\ref{contrib2}), (\ref{contrib3}) and (\ref{contrib4}).
Note that it is oscillatory in time and singular on the diagonal $x = y$.

Let us now turn to bounds on $\psi_\alpha^{int,r}$, defined by (\ref{defipsiintr}).
Contributions of (\ref{contrib1}), (\ref{contrib2}),  (\ref{contrib4}) and of the first term of  (\ref{contrib3}) 
decay like $\langle \alpha t \rangle^{-2}$ after integrations by
parts. 

The second term of (\ref{contrib3}), namely $Q_s(t,x,y)$,
leads to a term that decays like $\langle \alpha t \rangle^{-2}$ provided $x$ is away from $y$.
For $x$ close to $y$,  we use the description of $L$. 
As previously, $L_1$ and $L_2$ lead to a term in $O(\langle \alpha t \rangle^{-2})$, which ends the proof.
The bound on $\partial_y \psi_\alpha^{int,r}$ is similar.
\end{proof}


\subsubsection{Near an extremal value \label{nearextremal}}


\begin{lemma} \label{decay21}
If $c_r^0$ is  an extremal velocity $c_{extr}^l$, then
\beq \label{decay2}
| \psi_\alpha^{int,r}(t,y) | \lesssim \langle \alpha t \rangle^{-2} \theta( y - y_{extr}^l)
\eeq
and
\beq \label{decay2der}
| \partial_y \psi_\alpha^{int,r}(t,y) | \lesssim \langle \alpha t \rangle^{-1} \theta( y - y_{extr}^l).
\eeq
\end{lemma}

\begin{proof}
Let us again study $G_{\alpha}^{int,r}(t,x,y)$.
Let us  assume $x < y$ to fix the ideas, the case $x > y$ being similar.
We now assume that $c_r^0  = c_{extr}^l$ for some $1 \le l \le L$.
We first discuss the position of $x$ and $y$ with respect to $y_{extr}^l$.
We recall that $x \in I(x_p)$ and $y \in I(x_{p'})$ for some integers $p$ and $p'$.

Let us first investigate the case where neither $I(x_p)$ nor $I(x_{p'})$ contains $y_{extr}^l$, namely the case
where $x$ and $y$ are "away" from the extremal point $y_{extr}^l$.
Then, near $x$, Proposition \ref{propfirst2} gives the existence of two smooth solutions $\psi_{1,\alpha,c}^x$ 
and $\psi_{2,\alpha,c}^x$ with unit Wronskian.
We again decompose $\psi_{+,\alpha,c}(x)$ and $\psi_{-,\alpha,c}(y)$ using (\ref{decoco}). 

If $x > y_{extr}^l$ and $y > y_{extr}^l$ then both $\psi_{+,\alpha,c}$ and $\psi_{-,\alpha,c}$ are locally bounded. Thus, by construction,
$a_\pm^1$ and $a_\pm^2$ are bounded, and we are back to the situation of the previous section \ref{awayfrom},
hence (\ref{decay2}) holds true.

If $x < y_{extr}^l$ and $y > y_{extr}^l$, $\psi_{-,\alpha,c}$ is bounded near $y$. Thus $a_-^1$ and $a_-^2$ are
bounded. As $\psi_{+,\alpha,c}$ is of order $O( (c - c_{extr}^l)^{3/2})$,  $a_+^1$ and $a_+^2$  are
of the same order. Therefore, in this case also,
$a_\pm^1 a_\pm^2$ is bounded and we are  back to the previous section  \ref{awayfrom}.

If $x < y_{extr}^l$ and $y < y_{extr}^l$, then $a_-^1$ and $a_-^2$ are bounded by
$O( (c - c_{extr}^l)^{-3/2})$, whereas $a_+^1$ and $a_+^2$ are bounded
by $O( (c - c_{extr}^l)^{+3/2})$.
 Thus, again, $a_\pm^1 a_\pm^2$ is bounded and we are back to the previous section  \ref{awayfrom}.

The last case $x > y_{extr}^l$ and $y < y_{extr}^l$ is impossible since $x < y$ by assumption.

\medskip

We now turn to the opposite case $y_{extr}^l \in I(x_p) = I(x_{p'})$.
The case $y_{extr}^l \notin I(x_p)$, $y_{extr}^l \in I(x_{p'})$ and the case $y_{extr}^l \in I(x_p)$, $y_{extr}^l \notin I(x_{p'})$ are similar.

Close to $y_{extr}^l$, we note that $U_s''(x)$ does not vanish. Hence, using Rayleigh equation, we get 
\beq \label{firstexpression3}
{\psi_\pm(x) \over U_s(x) - c} = {\partial_x^2 \psi_\pm(x) - \alpha^2 \psi_\pm(x) \over U_s''(x) } .
\eeq
This leads to, for $x < y$,
$$
G^{int,r}_\alpha(t,x,y) 
=  \int_{\Gamma_\eps} \chi_r(c)  {\partial_x^2 \psi_{+,\alpha,c}(x) - \alpha^2 \psi_{+,\alpha,c}(x) \over U_s''(x) }
  \psi_{-,\alpha,c}(y) e^{- i \alpha c t} \, dc ,
$$
and symmetrically in $\pm$ for $x > y$.
We  define, for $\eps > 0$,
\beq \label{definitonGalpha4}
\widetilde G_\alpha^{int,r}(t,x,y) =  \int_{\Gamma_\eps} \chi_r(c)  \psi_{+,\alpha,c} ( x)  \psi_{-,\alpha,c} ( y  ) e^{- i \alpha c t} \, dc ,
\eeq
in such a way that, for $x \ne y$,
$$
G_\alpha^{int,r}(t,x,y) = (\partial_x^2 - \alpha^2) \widetilde G_\alpha^{int,r} .
$$
We have the following estimates on $\widetilde G_\alpha^{int,r}$.

\begin{lemma} \label{sins}
 If $|y| < \sigma$ and $|x| < \sigma$ with $\sigma > 0$ small enough, then
\beq \label{boundGG}
| \widetilde G_\alpha^{int,r} (t,x,y) | \lesssim {1 \over \langle \alpha t \rangle^2}  
{1 \over  |x - y_{extr}| +  |y - y_{extr}|} .
\eeq
\end{lemma}

\begin{proof}
We assume $\alpha = 1$, $c_{extr} = 0$, $y_{extr} = 0$ and $U_s''(y_{extr}) = 2$ to simplify the notations.
We detail the case $y > 0$.
Integrating by parts, we have
\beq \label{Grho2b}
\widetilde G_\alpha^{int,r}(t,x,y) =   - {1 \over \alpha^2 t^2}
 \int_{\Gamma_\eps} \partial_c^2 \Bigl[ \chi_r(c)  \psi_{+,\alpha,c} ( x)  \psi_{-,\alpha,c} ( y  ) \Bigr] e^{- i \alpha c t} \, dc .
\eeq
We split the integral defining $\widetilde G_\alpha^{int,r}(t,x,y)$ in three parts $M_1$, $M_2$, $M_3$, where
$M_1$ is the integral for $|c| \le 2^{-1} \min(|x|^2,|y|^2)$, $M_3$ the integral for $|c| \ge 2 \max(|x|^2,|y|^2)$ and $M_2$ the intermediate one.
We recall the inequalities (\ref{1}) and (\ref{2}).

Let us  study $M_1$ and first assume that $x < 0$ and $y > 0$. We have
$$
|  \psi_{+,\alpha,c} ( x)  \psi_{-,\alpha,c}(y) | \lesssim {|c|^{3/2} \over |x| |y|}.
$$
 Each time we differentiate one term, we loose a factor $|c|$, thus
$$
\Bigl | \partial_c^2 \Bigl[ \chi_r(c) \,   \psi_{+,\alpha,c} ( x)  \psi_{-,\alpha,c} (y) \Bigr] \Bigr| \lesssim 
{1 \over  |c|^{1/2} |x| \, |y|},
$$
which leads to (\ref{boundGG}), after integration in $c$.

Let us now investigate the case $0 < x < y$.
In this case, we have
$$
\Bigl|  \chi_r(c)  \psi_{+,\alpha,c}(x)  \psi_{-,\alpha,c}(y)  \Bigr| \lesssim  { | x |^2 \over y} .
$$
As a $c$ derivative leads to a loss of a factor $x$ or $y$, we have
$$
\Bigl|  \partial_c^2 \Big[ \chi_r(c)  \psi_{+,\alpha,c}(x)  \psi_{-,\alpha,c}(y)  \Bigr] \Bigr| \lesssim  {1 \over |y|} + {|x|^2 \over |y|^3}.
$$
The integral in $c$ is thus bounded by $|y|$, and hence by $(|x| + |y|)^{-1}$, provided  $\sigma$ is small enough.

Let us turn to the case  $x > y > 0$. We have
$$
\Bigl|  \chi_r(c)  \psi_{+,\alpha,c}(x)  \psi_{-,\alpha,c}(y)  \Bigr| \lesssim  { | y |^2 \over x},
$$
thus
$$
\Bigl|  \chi_r(c)  \psi_{+,\alpha,c}(x)  \psi_{-,\alpha,c}(y)  \Bigr| \lesssim  { |y|^2 \over |x|^3} + {1 \over |x|}.
$$
The corresponding integral is thus bounded by $|y|$, namely by $(|x| + |y|)^{-1}$, provided  $\sigma$ is small enough.

Let us turn to $M_3$. For $c \ge 2 \max(|x|^2,|y|^2)$, $\psi_-$ is of order $O(|c|^{-1/2})$ and $\psi_+$ is of order $O(| c |)$, thus
$$
\Bigl|   \chi_r(c)  \psi_{+,\alpha,c}(x)  \psi_{-,\alpha,c}(y)  \Bigr| \lesssim  |c|^{1/2}.
$$
Derivatives are a factor $|c|^{-1}$ larger. Thus, in this area,
$$
\Bigl|  \partial_c^2 [ \chi_r(c)  \psi_{+,\alpha,c}(x)  \psi_{-,\alpha,c}(y) ]  \Bigr| \lesssim  |c|^{-3/2},
$$
leading to  $| M_3 | \lesssim (|x| + |y|)^{-1}$.
The bounds on $M_2$  are similar.
 \end{proof}

We now turn to the bounds on $\psi_\alpha^{int,r}(t,y)$. If both $x$ and $y$ are in the support of $\chi_r$,
we have
\begin{equation*} \begin{split}
\psi_\alpha^{int,r}(t,y) &= \int_{\rit^+}  {\partial_x^2 \widetilde G_\alpha^{int,r}(t,x,y) 
- \alpha^2 \widetilde G_\alpha^{int,r}(t,x,y) \over U_s''(x) }  \omega_\alpha^0(x) \, dx
\\ & =  \int_{\rit^+} \widetilde G_\alpha^{int,r}(t,x,y) \widetilde \omega_\alpha^0(x) \, dx
- [\partial_x \tilde G_{\alpha}^{int,r} (t,x,y)]_{x=y} \partial_x \Bigl(  {\omega_\alpha^0(y) \over U_s''(y)} \Bigr),
\end{split} \end{equation*}
where
$$
\widetilde \omega_\alpha^0 =  \partial_x^2 \Bigl({\omega_\alpha^0 \over U_s''}  \Bigr) 
- \alpha^2 {\omega_\alpha^0 \over U_s''} 
$$
and  $[\partial_x \tilde G_{\alpha}^{int,r}(t,x,y)]_{x=y}$ is the jump of $\partial_x \tilde G_\alpha^{int,r}(t,x,y)$ at $x = y$.
Using Lemma \ref{sins}, we have
$$
  \Bigl| \int_{\rit^+} \widetilde G_\alpha^{int,r}(t,x,y) \widetilde \omega_\alpha^0(x) \, dx \Bigr|
 \lesssim   {1 \over \langle \alpha  t \rangle^2}\int  {| \widetilde \omega_\alpha^0(x)| \over |x| + |y|}  \, dx
\lesssim   {|\log y | \over \langle \alpha t \rangle^2} .
$$
It remains to bound the jump term. As $W[\psi_{+,\alpha,c},\psi_{-,\alpha,c}] = 1$, we have
$$
[\partial_x \widetilde G_\alpha^{int,r}](t,y) = 
  \int_{\Gamma(0)} \chi_r(c)   e^{- i \alpha c t} \, dc = O(\langle \alpha t \rangle^{-2}).
$$
Bounds on $\partial_y \psi_\alpha^{int,r}$ are similar, but only lead to a decay in $\langle \alpha t \rangle^{-1}$.
\end{proof}


\subsection{Study of $G_\alpha^b$ and $\psi_\alpha^b$}


We now turn to the boundary layer parts $G_\alpha^b$ of $G_\alpha$ and $\psi_\alpha^b$ of $\psi_\alpha$.

\begin{lemma}
Under the assumptions (A1), (A2), (A3) and (A4),
$$
\psi_\alpha^b  = \psi_\alpha^{b,modes} + \psi_\alpha^{b,decay}
$$
where $\psi_\alpha^{b,modes}$ is the projection on $\sigma_P \cap \{c, \Im c \ge 0 \}$,
and 
\beq \label{decaypsib}
\| \psi_\alpha^b(t,\cdot) \|_{L^\infty} \lesssim \langle \alpha t \rangle^{-1} .
\eeq
\end{lemma}

\Remarks
Note that the boundary layer part of the velocity decays like the horizontal part of the velocity, namely like $\langle \alpha t \rangle^{-1}$ and not faster.
Moreover, $\psi_\alpha^{b,modes}$ is an explicit finite sum of exponentially increasing or oscillatory eigenmodes.

\begin{proof}
We recall that
$$
G_{\alpha,c}^{b} (x,y) =  -     { \psi_{+,\alpha,c}(0)  \over  \psi_{-,\alpha,c}(0) } {\psi_{-,\alpha,c}(x) \over U_s(x) - c}  \psi_{-,\alpha,c}(y),
$$
which leads to
$$
\psi_{\alpha,c}^b(y) = \int_{\rit^+} G_\alpha^b(x,y) \omega_\alpha^0(x) \, dx = {L_{\alpha,c}(\omega_\alpha^0) \over \psi_{-,\alpha,c}(0)} \psi_{-,\alpha,c}(y)
$$
where
$$
L_{\alpha,c}(\omega_\alpha^0) = - \int_{\rit^+} {\psi_{+,\alpha,c}(0) \psi_{-,\alpha,c}(x) \over U_s(x) - c} \omega_\alpha^0(x) \, dx.
$$
A simple computation shows that $\bar \psi_{-,\alpha,c} / (U_s(x) - \bar c)$ is in fact an eigenmode for the adjoint of Rayleigh operator, thus
$L_{\alpha,c}(\omega_\alpha^0) \psi_{-,\alpha,c}$ can be interpreted as the projection on the eigenmode $\psi_{-,\alpha,c}$ of Rayleigh operator.
We have  to study
$$
\psi^b_\alpha(t,y) = \int_\Gamma \psi_{\alpha,c}^b(y) e^{-i \alpha c t} dc .
$$
We note that the integrand of $\psi^b_{\alpha}$ is holomorphic outside $Range(U_s)$,
can be extended by continuity to $Range(U_s)$,
has poles at the various eigenvalues $c \in \sigma_P \cap \{c, \Im c > 0 \}$, and is not smooth at the embedded eigenvalues and at  $c = 0$.
We first move the contour $\Gamma$ "below" the eigenvalues. This leads to
\beq \label{decompGb}
\psi^b_\alpha = \psi^{b,unstable,+}_\alpha + \widetilde \psi^b_\alpha,
\eeq
where
$$
\psi^{b,unstable,+}_\alpha(t,y) = - 2 i \pi  \sum_{c \in \sigma_P, \Im c > 0}  
{L_{\alpha,c}(\omega_\alpha^0) \over \partial_c \psi_{-,\alpha,c}(0)} \psi_{-,\alpha,c}(y) e^{-i \alpha c t}
$$
is the projection on unstable modes and 
$$
\widetilde \psi^b_\alpha(t,y) = \int_{\Gamma_\eps} \psi_{\alpha,c}^b(y) e^{-i \alpha c t} d\alpha .
$$
We choose $\eps$  small enough, so that all the unstable eigenvalues are "above" $\Gamma_\eps$.
As previously, we introduce $\chi_r$, a partition of unity and $\widetilde \psi_{\alpha}^{b,r}$, defined by
$$
\widetilde \psi_\alpha^{b,r}(t,y) = \int_{\Gamma_\eps} \chi_r(c) \psi_{\alpha,c}^b(y) e^{-i \alpha c t} dc .
$$
If the support of $\chi_r$ contains no singularity of $L_{\alpha,c}(\omega_\alpha^0)$, then $| \widetilde \psi^b_\alpha(t,y) | = O(\langle \alpha t \rangle^{-2})$.
If it contains an embedded eingenvalue $c_{embed}^l$, then
$$
\lim_{\eps \to 0} \widetilde \psi_\alpha^{b,r}(t,y) = \psi_\alpha^{b,unstable,0}(t,y) + O(\langle \alpha t \rangle^{-1})
$$
where
$$
\psi_\alpha^{b,unstable,0}(t,y) = 
-  2 i \pi 
 {L_{\alpha,c_{embed}^l}(\omega_\alpha^0) \over \partial_c \psi_{-,\alpha,c_{embed}^l}(0)} \psi_{\alpha,c_{embed}^l}(y) 
e^{-i \alpha c_{embed}^l t} 
$$
can be interpreted as the projection  on $\sigma_{embed}^l \in \sigma_P \cap \{ c , \Im c = 0 \}$.

If the support of $\chi_r$ contains an extremal point $c_{extr}^l$, then, near this extremal point,
$\psi_{+,\alpha,c}(0)  \psi_{-,\alpha,c}^{-1}(0)$ is of order $O((c - c_{extr}^l)^3)$ and $\psi_{-,\alpha,c}(x)$ is bounded when $x$ is bounded.
Thus $L_{\alpha,c}(\omega_\alpha^0)  \psi_{-,\alpha,c}^{-1}(0)$ is of order $O((c - c_{extr}^l)^3)$.
In particular,
$\psi_{\alpha,c}^b(y)$ is of order $O((c - c_{extr}^l)^{3/2})$ if $y < y_{extr}^l$ and of order $O((c - c_{extr}^l)^{3})$ if $y > y_{extr}^l$.
In any case, it goes to $0$ as $c$ goes to $c_{extr}^l$.
Thus, there is no singularity at this point. On the contrary, $\psi_\alpha^{b,r}(t,y)$ is small, which is  linked to the depletion property
(see next section).

We define
$$
\psi_\alpha^{modes} = \psi_\alpha^{b,modes} = \psi_\alpha^{b,unstable,+} + \psi_\alpha^{b,unstable,0}
$$ 
and $\psi_\alpha^{b,decay}$ to be the sum of all the other terms, in such a way that 
$$
\psi_{\alpha}^b = \psi_\alpha^{b,decay} + \psi_\alpha^{b,modes}.
$$
Note that the modes only come from the boundary part $\psi_\alpha^b$.
We further define
$$
\psi_\alpha^{decay} = \psi_\alpha^{int} + \psi_\alpha^{b,decay}.
$$
 We note that
$$
 | \psi_\alpha^{b,decay} |  \lesssim \langle \alpha t \rangle^{-1},
$$
which ends the proof.
\end{proof}


\section{Behaviour of the  vorticity}



\subsection{Vorticity depletion}


We now turn to the proof of the vorticity depletion. 
We recall that $\omega_\alpha^{int}$ is the vorticity of $\psi_\alpha^{int}$
and that $\omega_\alpha^b$ is the vorticity of $\psi_\alpha^b$. We decompose $\omega_\alpha^b$ in
$\omega_\alpha^{b,modes} + \omega_\alpha^{b,decay}$. We define
$$
\zeta_\alpha^{int} = - (U_s - c) \omega_\alpha^{int}, \qquad
\zeta_\alpha^{b,decay} = - (U_s - c) \omega_\alpha^{decay}
$$
and
$$
\zeta_\alpha^{b,modes} = - (U_s - c) \omega_\alpha^{modes}.
$$
Using Rayleigh equation, we have
$$
\zeta_{\alpha}^{int}(y,c) = i \alpha U_s''(y) \psi_{\alpha,c}^{int}(y)  -  \omega_\alpha^0(y),
$$
$$
\zeta_{\alpha}^{b,decay}(y,c) = i \alpha U_s''(y) \psi_{\alpha,c}^{b,decay}(y)  
$$
and
$$
\zeta_{\alpha}^{b,modes}(y,c) = i \alpha U_s''(y) \psi_{\alpha,c}^{b,modes}(y)  .
$$
We next study $\zeta_\alpha^{int}(y,c)$ and $\zeta_\alpha^{b,decay}(y,c)$.

\begin{proposition}   \label{vanish} 
For any $1 \le l \le L$, we have
\beq \label{vanish1}
\zeta_{\alpha}^{int}(y_{extr}^l,c_{extr}^l) = 0.
\eeq
Moreover, for $y$ close to $y_{extr}^l$,
\beq \label{cancel}
| \zeta_{\alpha}^{int}(y,U_s(y))   
| \lesssim | y - y_{extr}^l |^2   \Bigl| \log | y - y_{extr}^l  | \Bigr|  
\eeq
and
\beq \label{depletionzetab}
| \zeta_\alpha^b(y,U_s(y)) | \lesssim | y - y_{extr}^l |^2.
\eeq
\end{proposition}

\Remark
We do not need to study $\omega_\alpha^{b,modes}$ since it is explicit, as a sum of modes. 
Moreover, $\omega_\alpha^{b,modes}$ in  general does not satisfy (\ref{cancel}).

\begin{proof}

Near an extremal value $c_{extr}^l$, 
we already know that $\psi_\alpha^{b,decay}$ is of order  $(c - c_{extr}^l)^{3/2}$, which gives (\ref{depletionzetab}).

We now turn to $\zeta_\alpha^{int}$.
Up to a translation of $y$ and $c$, we assume that $y_{extr}^l = 0$ and $c_{extr}^l = 0$ to alleviate the formulas.
Note that $U_s''(0) \ne 0$, thus, on some interval $[-\eps,+\eps]$, $U_s''(y)$ is bounded away from $0$. Let
$\chi(x)$ be a smooth function supported in $[-\eps,+\eps]$,  which equals $1$ on $[-\eps/2,+\eps/2]$.
Then, for $\Im c > 0$,
\beq \label{expepe}
\psi_{\alpha,c}^{int}(y) = \int \chi(x) G_{\alpha,c}^{int}(x,y) \omega_\alpha^0(x) \, dx 
+ \int (1 - \chi(x)) G_{\alpha,c}^{int}(x,y) \omega_\alpha^0(x) \, dx.
\eeq
Let  $N_1$ and $N_2$ be respectively the first and second integrals.
In $N_1$, the integrand is as singular as $\psi_{\pm,\alpha,c} (x)/ (U_s(x) - c)$. 
We now follow the lines of section \ref{awayfrom}. To describe this singularity, as $U_s''(x)$ does not vanish, we rewrite 
it under the form
\beq \label{firstexpression2}
{\psi_{\pm,\alpha,c}(x) \over U_s(x) - c} = {\partial_x^2 \psi_{\pm,\alpha,c}(x) - \alpha^2 \psi_{\pm,\alpha,c}(x) \over U_s''(x) } .
\eeq
For $\Im c > 0$, this leads to
$$
N_1 = \int_{-\eps}^{+\eps} \chi(x) {(\partial_x^2 - \alpha^2) \widetilde G_{\alpha,c}(x,y) \over U_s''(x)} \omega_\alpha^0(x) \, dx,
$$
where
\beq \label{tildeGac}
 \widetilde G_{\alpha,c}^{int}(x,y) =
 \Bigl\{ \begin{array}{c}
 \psi_{-,\alpha,c}(y) \psi_{+,\alpha,c}(x) \qquad \hbox{if} \qquad y > x \cr
 \psi_{-,\alpha,c}(x) \psi_{+,\alpha,c}(y) \qquad \hbox{if} \qquad y < x . \cr
\end{array} 
\eeq
Integrating by parts, and using that $\partial_x \widetilde G_{\alpha,c}$ has a jump at $x = y$ which equals $-1$, we obtain
$$
N_1 = - {i \over \alpha}\chi(y) \omega_\alpha^0(y) + \int_{-\eps}^{+\eps} \widetilde G_{\alpha,c}(x,y) \theta(x) \, dx
$$
where
$$
\theta(x) = (\partial_x^2 - \alpha^2) \Bigl({ \chi \omega_\alpha^0 \over U_s''} \Bigr).
$$
We have
\beq \label{formula}
 \int_{-\eps}^{+\eps} \widetilde G_{\alpha,c}(x,y) \theta(x) \, dx 
 = \psi_{-,\alpha,c}(y) \int_{-\eps}^y \psi_{+,\alpha,c}(x) \theta(x) \, dx
 \eeq
 $$
 + \psi_{+,\alpha,c}(y) \int_y^{+ \eps} \psi_{-,\alpha,c}(x) \theta(x) \, dx.
 $$
 Let $I_1$ and $I_2$ be respectively the first and  second integrals on the right hand side of the above equality.
 Let us assume that $y > y_+$ to fix the ideas, the cases $y < y_-$ and $y_- < y < y_+$ being similar.
 
 Let us first bound $I_1$. 
For $x < y_-$,  $|\psi_{+,\alpha,c}(x) | \lesssim | c|^{3/2} | x|^{-1}$, 
thus the corresponding contribution to the integral is bounded by $|c|^{3/2} \log |y|$. 
Between $y_-$ and $y_+$, $| \psi_{+,\alpha,c}(x) | \lesssim |c|$, leading to a contribution of order $|c|^{3/2}$ to the integral.
For  $y_+ < x <y$, $| \psi_{+,\alpha,c}(x) | \lesssim |x|^2$, thus the integral between $y_+$ and $y$ is bounded by $|y|^3$.
As $|\psi_{-,\alpha,c}(y) | \lesssim |y|^{-1}$,
$$
| I_1 | \lesssim | \psi_{-,\alpha,c}(y) | \Bigl[ | c |^{3/2} |\log |y|| + |c|^{3/2}  \Bigr] 
\lesssim |c|^{3/2} |y|^{-1} |\log |y| | + |c|^{3/2} |y|^{-1} + |y|^2.
$$
Now, as $y > 0$, $| \psi_{+,\alpha,c}| \lesssim y^2$, and for $x > y$, $| \psi_{-,\alpha,c}(x) | \lesssim |x|^{-1}$, thus
$$
| I_2 | \lesssim y^2 \log y.
$$
 For $c = U_s(y)$, namely of order $y^2$, all these terms are bounded by $y^2 \log y$.

Let us turn to $N_2$.
For $x < - \eps/2$, $\psi_{+,\alpha,c}(x)$ is of order $O(c^{3/2})$, thus the computations are similar,
 and for $x > \eps$, $|\psi_{+,\alpha,c} (y) | \lesssim   y^2$
and $| \psi_{-,\alpha,c}(x) |$ is bounded.
Combining these estimates, we get
$$
\Bigl| \psi_{\alpha,c}^{int}(y) + {i \over \alpha} \omega_\alpha^0(y)  \Bigr| \lesssim y^2 | \log y |,
$$
which is exactly (\ref{cancel}). This implies (\ref{vanish1}) and ends the proof.
\end{proof}


\subsection{Asymptotic behavior of $\omega_\alpha^{int}(t,y)$}


We now turn to the asymptotic behaviour of the vorticity $\omega_\alpha^{int}(t,y)$.  
The proof follows the lines of \cite{Bouchet-Morita-2010}. Using Rayleigh equation, we have
\begin{equation*} \begin{split}
\omega_\alpha^{int}(t,x) &= {1 \over 2 i \pi} \int_\Gamma {\zeta_\alpha^{int}(y,c) \over U_s(y) - c} e^{-i \alpha c t} dc
\\ &= {1 \over 2 i \pi} \int_\Gamma {\zeta_\alpha^{int}(y,U_s(y)) \over U_s(y) - c} e^{-i \alpha c t}  dc
\\ &+ {1 \over 2 i \pi} \int_\Gamma {\zeta_\alpha^{int}(y,c) - \zeta_\alpha^{int}(y,U_s(y)) \over U_s(y) - c} e^{-i \alpha c t}  dc
\\& = \omega_\alpha^\infty(y)  e^{- i \alpha U_s(y) t} + R(t,y),
\end{split} \end{equation*}
where we define $\omega_\alpha^\infty$ by
$$
\omega_\alpha^\infty(y) = \zeta_\alpha^{int}(y,U_s(y))
$$ 
and  $R(t,y)$ is the  integral term. Let us now bound $R(t,y)$. 
We note that
$$
\zeta_\alpha^{int}(y,c) - \zeta_\alpha^{int}(y,U_s(y))  = i \alpha U_s''(y) \Bigl( \psi_{\alpha,c}^{int}(y) - \psi_{\alpha,U_s(y)}^{int}(y) \Bigr) .
$$
The speed of decay of $R(t,y)$ thus depends on the smoothness of $\psi_{\alpha,c}^{int}(y)$ viewed as a function of $c$.

We note that $\psi_{\alpha,c}^{int}(y)$ is a $C^\infty$ function of $c$ provided $c$ is not an extremal value.
If $c$ is close to an extremal value $c_{extr}^l$, $\psi_{\alpha,c}^{int}(y)$ is still smooth provided $y$ is away from  $y_{extr}^l$
Hence, we just need to focus on the smoothness of $\psi_{\alpha,c}^{int}(y)$ for $c$ close to $c_{extr}^l$ and $y$ close to $y_{extr}^l$.
Up to a translation, we may assume $c_{extr}^l = 0$ and $y_{extr}^l = 0$.

In this latest case, we can use the explicit formula (\ref{formula}).
Moreover,  $\partial_c^2 \widetilde G_{\alpha,c}(x,y)$ has already been computed during the proof of Lemma \ref{sins}, where we proved
$\| \partial_c^2 \widetilde G_{\alpha,c}(x,y) \|_{L^1_c} \lesssim (|x| + |y|)^{-1}$. Using (\ref{expepe}), we obtain that
$\| \psi_{\alpha,c}^{int}(x,y) \|_{L^1_c} \lesssim | \log |y| |$. This implies that
$$
| R(t,y) | \lesssim {| \log |y| | \over \langle \alpha t \rangle} .
$$
This gives (\ref{omegaremain}) and ends the study of $\omega^{remain}(t,y)$.


\section{Appendix}


We first recall the classical Sokhotski–Plemelj theorem.

\begin{theorem}  \label{Plemelj} ( Sokhotski-Plemelj theorem) \\
Let $\phi$ be a smooth function defined on the real line, with compact support. Then
$$
\lim_{\eps \to 0^+} \int_{\rit} {\phi(x) \over x - i \eps} dx = i \pi \phi(0) + PV \int_\rit {\phi(x) \over x} dx,
$$ 
where $PV$ denotes the principal value of the integral.
\end{theorem}

\begin{proof}
We write
$$
{1 \over x - i \eps} = {x + i \eps \over x^2 + \eps^2} .
$$
We then have, as $\eps \to 0^+$,
\beq \label{imag}
\int_{\rit} {\eps \over x^2 + \eps^2} \phi(x) \, dx = \int_\rit {\phi(\eps y) \over y^2 + 1} dy \to \pi \phi(0).
\eeq
Moreover, as $\phi$ has a compact support, $\phi(x) = 0$ outside some interval of the form $[-A,+A]$. Let $\psi(x)$ be defined by
$\phi(x) = \phi(0) + x \psi(x)$. Then $\psi$ is smooth and we have
\beq \label{real}
\int_\rit {x \over x^2 + \eps^2} \phi(x) dx = \int_{-A}^{+A} {x \over x^2 + \eps^2}  \phi(0) dx
+ \int_{-A}^{+A} {x^2 \over x^2 + \eps^2} \psi(x) dx.
\eeq
The first term equals $0$ and the limit of the second is exactly the principal value.
\end{proof}

From Theorem \ref{Plemelj}, we can get the following results.

\begin{proposition}
Let $\phi$ be a smooth function defined on the real line with compact support. Then it holds that
\beq \label{ple1}
\lim_{\eps \to 0^+} \int_{\rit} {\phi(x)  e^{i x t} \over x - i \eps} dx = 2 i \pi \phi(0) + O(\langle t \rangle^{-2}) ,
\eeq
\beq \label{ple2}
\lim_{\eps \to 0^+} \int_{\rit} \phi(x)  e^{i x t}  \log(x - i \eps) dx = { 2 \pi \over t} \phi(0) + O(\langle t \rangle^{-2}) ,
\eeq
\beq \label{ple3}
\lim_{\eps \to 0^+} \int_{\rit} \phi(x)  e^{i x t}  (x - i \eps) \log(x - i \eps) dx = O(\langle t \rangle^{-2}) .
\eeq\end{proposition}

\begin{proof}
We first prove (\ref{ple1}). First (\ref{imag}) leads to a first term $i \pi \phi(0)$.
Then
$$
\int_\rit {x \over x^2 + \eps^2} \phi(x) e^{i x t} dx = \phi(0) \int_\rit {x e^{i x t} \over x^2 + \eps^2} dx + \int {x^2 \over x^2 + \eps^2} \psi(x) e^{i x t} dx.
$$
As $\eps$ goes to $0$, the last integral converges to the integral of $\psi(x) e^{ixt}$ which is in particular smaller than $\langle t \rangle^{-2}$.
Moreover,
$$
\lim_{\eps \to 0^+} \int_\rit {x e^{i x t}  \over x^2 + \eps^2} dt = 2 i \lim_{\eps \to 0^+} \int_\rit {x \over x^2 + \eps^2} \sin(x) dx = i \pi,
$$
which gives the formula.
Equality (\ref{ple2}) can be deduced from (\ref{ple1}) after an integration by parts, and  (\ref{ple2}) implies  (\ref{ple3}) after
an integration by parts.
\end{proof}


{\bf Acknowledgments}  D. Bian is supported by NSF of China under the Grant 12271032.



\end{document}